\documentclass[english,10pt]{article}
\usepackage[papersize={19cm,25cm},body={15cm,21cm},centering]{geometry}
\usepackage[T1]{fontenc}
\usepackage[latin9]{inputenc}
\usepackage{amsmath}
\usepackage{amssymb}
\usepackage{url}
\usepackage{babel}
 \usepackage{graphicx}
\usepackage{epsfig,subfigure,color}
\usepackage{graphics}
 \usepackage{graphicx}
\usepackage{amsmath,amsfonts,amsthm,latexsym}
\usepackage{mathrsfs,amsbsy}
\usepackage{multirow}
\usepackage{epstopdf}
\usepackage{algorithm}
\usepackage{algorithmic}
\usepackage{cite}

\newtheorem{theorem}{Theorem}
\newtheorem{lemma}{Lemma}
\newtheorem{property}{Property}
\newtheorem{example}{Example}
\newtheorem{remark}{Remark}

\theoremstyle{definition}
\newtheorem{defn}{Definition}
\usepackage{bm}

\usepackage{rotating}  

 \usepackage{color}
 \usepackage{url}

\DeclareMathOperator{\bcirc}{bcirc}
\DeclareMathOperator{\Bdiag}{Bdiag}
\DeclareMathOperator{\Diag}{Diag}
\DeclareMathOperator{\fold}{fold}
\DeclareMathOperator{\unfold}{unfold}
 
\begin{document}
\title{Block Diagonalization of Quaternion Circulant Matrices with Applications}
\date{}
\author{
Junjun Pan \qquad  Michael K. Ng\thanks{Department of Mathematics,
Hong Kong Baptist University.
Emails: junjpan@hkbu.edu.hk, michael-ng@hkbu.edu.hk. M. Ng's research is supported
in part by Hong Kong Research Grant Council GRF 
17201020, 17300021, C7004-21GF and Joint NSFC-RGC N-HKU76921.
} 
}
\maketitle

\begin{abstract}
It is well-known that a complex circulant matrix can be diagonalized by a discrete 
Fourier matrix with imaginary unit $\mathtt{i}$. 
The main aim of this paper 
is to demonstrate that a quaternion circulant matrix cannot be 
diagonalized by a discrete
quaternion Fourier matrix with three imaginary units $\mathtt{i}$,
$\mathtt{j}$ and $\mathtt{k}$. Instead, a quaternion circulant matrix can be 
block-diagonalized into 1-by-1 block and 2-by-2 block matrices by permuted  
discrete quaternion Fourier transform matrix. With such a block-diagonalized form, 
the inverse of a quaternion circulant matrix can be determined efficiently similar to the inverse of a complex circulant matrix. 
We make use of this block-diagonalized form 
to study quaternion tensor singular value 
decomposition of quaternion tensors where the entries are quaternion numbers. The applications including computing the inverse of a quaternion circulant matrix, and solving quaternion Toeplitz system
arising from linear prediction of quaternion signals  are employed to validate the efficiency of our proposed block diagonalized results. A numerical example of color video as third-order quaternion tensor is employed to validate the effectiveness of quaternion tensor singular value decomposition. 

 
\end{abstract}

\vspace{2mm}

\vspace{3mm}

\section{Introduction}

Let us start with notations used throughout this paper.  The real number field, the complex number field and the quaternion algebra are defined by $\mathbb{R}$, $\mathbb{C}$ and $\mathbb{Q}$ respectively. Unless otherwise specified,
lowercase letters represent real
numbers, for example, $a\in \mathbb{R}$. The bold lowercase letters represent real vectors, such as, $\mathbf{a}\in \mathbb{R}^n$. 
Real matrices are denoted by bold capital letters, 
like $\mathbf{A} \in \mathbb{R}^{m\times n}$. The numbers, vectors, and matrices under the quaternion field are represented by the corresponding symbols with breve, for example $ \breve{a} \in \mathbb{Q}$,  $\mathbf{\breve{a}}\in \mathbb{Q}^n$ and $\mathbf{\breve{A}}\in \mathbb{Q}^{m\times n}$.

The quaternions field $\mathbb{Q}$ is generally represented in the following Cartesian form,
$$ 
\breve{q}=q_0+\mathtt{i}q_{1}+\mathtt{j}q_2+\mathtt{k}q_3,
$$
where $q_0,q_1,q_2,q_3 \in \mathbb{R}$, and $\mathtt{i},\mathtt{j},\mathtt{k}$ are imaginary units such that 
 $$\mathtt{i}^2=\mathtt{j}^2=\mathtt{k}^2=-1,  \quad \mathtt{i}\mathtt{j}=-\mathtt{j}\mathtt{i}=\mathtt{k}, \quad \mathtt{j}\mathtt{k}=-\mathtt{k} \mathtt{j}= \mathtt{i},\quad \mathtt{k}\mathtt{i}=-\mathtt{i}\mathtt{k}=\mathtt{j},\quad \mathtt{i}\mathtt{j}\mathtt{k}=-1. $$
Any quaternion $\breve{q}$ can be simply written as $\breve{q}=\textit{Re}~ \breve{q}+ \textit{Im}~\breve{q}$ with real component $\textit{Re}~\breve{q}= q_0$ and imaginary component $\textit{Im}~\breve{q}= \mathtt{i}q_1+ \mathtt{j}q_2+\mathtt{k}q_3$. We call a quaternion $\breve{q}$ pure quaternion if its real component $\textit{Re}~\breve{q}=0$.  The quaternion conjugate $\bar{\breve{q}}$ and the modulus $|\breve{q}|$ of $\breve{q}$ are defined as 
$$\bar{\breve{q}}\doteq\textit{Re}~\breve{q} -\textit{Im}~\breve{q}=q_0-\mathtt{i} q_1-\mathtt{j}q_2-\mathtt{k}q_3,\quad  |\breve{q}|\doteq \sqrt{\breve{q}\bar{\breve{q}}}=\sqrt{q_0^2+q_1^2+q_2^2+q_3^2}.$$ 
A quaternion is a unit quaternion if its modulus equals to 1, i.e., $|\breve{q}|=1$. 
The dot product of two quaternions $\breve{a}=a_0+a_1 \tt{i}+a_2 \tt{j}+a_3\tt{k}$ and $\breve{b}=b_0+b_1 \tt{i}+b_2 \tt{j}+b_3\tt{k}$ is defined as 
$\breve{a} \cdot \breve{b} = a_0b_0+a_1b_1+a_2b_2+a_3b_3$.
Similarly, for 
quaternion matrix $\breve{\mathbf{Q}}=(\breve{q}_{uv})\in \mathbb{Q}^{m\times n}$, we denote its transpose $\breve{\mathbf{Q}}^T=(\breve{q}_{vu})\in \mathbb{Q}^{n\times m}$ and its conjugate-transpose $ \breve{\mathbf{Q}}^*=(\bar{\breve{q}}_{vu})\in \mathbb{Q}^{n\times m}$.

\subsection{Complex Circulant Matrices}

A complex circulant matrix $\mathbf{C}_0 + \mathbf{C}_1 \mathtt{i} \in \mathbb{C}^{n\times n}$ has the following form,
$$
\mathbf{C}_0 + \mathbf{C}_1 \mathtt{i} =\left(
           \begin{array}{ccccc}
            c_{0}^{(0)} + c_{1}^{(0)} \mathtt{i} &
            c_{0}^{(n-1)} + c_{1}^{(n-1)} \mathtt{i} &\cdots&
      c_{0}^{(2)} + c_1^{(2)} \mathtt{i} &
            c_{0}^{(1)} + c_1^{(1)} 
            \mathtt{i} \\
            c_{0}^{(1)} + c_1^{(1)} \mathtt{i} &
            c_{0}^{(0)} + c_1^{(0)} \mathtt{i} &
            \cdots&  & 
            c_{0}^{(2)} + c_1^{(2)} \mathtt{i} \\
            \vdots& \ddots &\ddots& \ddots &\vdots\\
						c_{0}^{(n-2)} + c_1^{(n-2)} \mathtt{i} &    
					&\cdots&
					c_{0}^{(0)} + c_1^{(0)} \mathtt{i} &
            c_{0}^{(n-1)} + c_1^{(n-1)} \mathtt{i} \\
            c_{0}^{(n-1)} + c_1^{(n-1)} \mathtt{i} &  
					c_{0}^{(n-2)} + c_1^{(n-2)} \mathtt{i} &\cdots&
					c_{0}^{(1)} + c_1^{(1)} \mathtt{i} &
            c_{0}^{(0)} + c_1^{(0)} \mathtt{i}
           \end{array}
         \right),
$$
simply denoted as $\mathbf{C}_{0} + \mathbf{C}_1 \mathtt{i} =
\text{circ}(\mathbf{c}_0 + 
\mathbf{c}_1 \mathtt{i} )$, with 
$$
\mathbf{c}_0
+ 
\mathbf{c}_1 \mathtt{i}
=[c_0^{(0)} + c_1^{(0)} \mathtt{i},
c_0^{(1)} + c_1^{(1)} \mathtt{i},\cdots,
c_0^{(n-1)} + c_1^{(n-1)} \mathtt{i}].
$$
Note that each row vector is rotated one element to the right relative to the preceding row vector. 
The eigenvectors of an $n\times n$ complex circulant matrix are the columns of $ \mathbf{F}^*_{\mathtt i}$ (or $ \mathbf{F}_{\mathtt i}$), where $ \mathbf{F}_{\mathtt i}$ is the discrete Fourier transform matrix given by
$$
[\mathbf{F}_{\mathtt i} ]_{uv}= \frac{1}{\sqrt{n}}
\exp \left ( \frac{-2\pi \mathtt{i}} {n} \right )^{uv}, \quad u,v=0,\cdots,n-1,
$$
that is,
$$
(\mathbf{C}_0 + \mathbf{C}_1 \mathtt{i})
 \mathbf{F}^*_{\mathtt i} = 
 \mathbf{F}^*_{\mathtt i}
(\mathbf{\Lambda}_0 + \mathbf{\Lambda}_1 \mathtt{i}), \quad {\rm or} 
\quad 
\mathbf{C}_0 + \mathbf{C}_1 \mathtt{i} =
 \mathbf{F}^*_{\mathtt i} (\mathbf{\Lambda}_0 + \mathbf{\Lambda}_1 \mathtt{i})
 \mathbf{F}_{\mathtt i},
$$
where $\mathbf{\Lambda}_0$ and $\mathbf{\Lambda}_1$ are 
diagonal matrices.
In other words, $(\mathbf{C}_0 + \mathbf{C}_1 \mathtt{i})$
can be diagonalized by 
the discrete Fourier transform matrix with imaginary unit $\mathtt{i}$.
In the following discussion, we refer
the above discrete Fourier transform matrix 
to be $ \breve{\mathbf{F}}_{\mathtt{i}}$ 
as it is described using quaternion numbers but associated with
imaginary unit $\mathtt{i}$ in the complex field only.

\subsection{The Contribution}

In this paper, we are interested in quaternion circulant matrices 
$\breve{\mathbf{C}} \in \mathbb{Q}^{n\times n}$ has the following form 
$$
\breve{\mathbf{C}} = 
\text{circ}(
\breve{\mathbf{c}}) = 
\left(
           \begin{array}{ccccc}
            \breve{c}^{(0)} &
            \breve{c}^{(n-1)} &\cdots & \breve{c}^{(2)} &
            \breve{c}^{(1)} \\
            \breve{c}^{(1)} &
            \breve{c}^{(0)} & \cdots 
            &  &
            \breve{c}^{(2)} \\
            \vdots& \ddots &\ddots& \ddots & \vdots\\
						\breve{c}^{(n-2)} & 
            & \cdots & 
            \breve{c}^{(0)} & \breve{c}^{(n-1)} \\
            \breve{c}^{(n-1)} &
            \breve{c}^{(n-2)} &\cdots& \breve{c}^{(1)} & 
            \breve{c}^{(0)} 
           \end{array}
         \right),
$$
with $\breve{\mathbf{c}}=[
\breve{c}^{(0)},\breve{c}^{(1)},\cdots,\breve{c}^{(n-1)}]$.
We would ask whether a quaternion circulant matrix $\breve{\mathbf{C}}$ can
be diagonalized by $\breve{{\bf F}}_{\mathtt{i}}$ or other discrete 
quaternion Fourier transform matrices $\breve{\mathbf F}_{\breve{\mu}}$ ?
Here $\breve{\mathbf F}_{\breve{\mu}}$
is defined as follows \cite{ell2014quaternion,flamant2019time,flamant2017spectral,Pei}:
\begin{equation} \label{qfft}
[\breve{\mathbf{F}}_{\breve{\mu}}]_{uv}= \frac{1}{\sqrt{n}}
\exp \left ( \frac{-2\pi \breve{\mu}} {n} \right )^{uv}, \quad u,v=0,\cdots,n-1,
\end{equation}
where $\breve{\mu}$ is a pure quaternion 
$$
\breve{\mu} = \mu_1 \mathtt{i} + \mu_2 \mathtt{j} + \mu_3 \mathtt{k}
$$ 
with $| \breve{\mu} | = 1$. 
For example, when $\breve{\mu} = \mathtt{i}$, 
$\breve{\mathbf F}_{\breve{\mu}}$ is equal to $\breve{\mathbf{F}}_{\mathtt{i}}$.
In signal processing, 
$\breve{\mu}$ is usually set to be $\frac{1}{\sqrt{3}} \mathtt{i} +
\frac{1}{\sqrt{3}} \mathtt{j} + 
\frac{1}{\sqrt{3}} \mathtt{k}$, see \cite{Pei}.
In general, due to the non-commutative property of quaternary field, the answer of the above question is negative, i.e., 
$\breve{\mathbf{C}}$ cannot be diagonalized by 
$\breve{{\bf F}}_{\mathtt{i}}$ or $\breve{{\bf F}}_{\breve{\mu}}$. 


Though $\breve{\mathbf{C}}$ cannot be diagonalized by 
$\breve{{\bf F}}_{\mathtt{i}}$ or $\breve{{\bf F}}_{\breve{\mu}}$, it can still be transformed into a simple structure. More precisely, 
the structure of 
$\breve{\mathbf F}_{\breve{\mu}} \breve{\mathbf{C}}
\breve{\mathbf F}^*_{\breve{\mu}}$ is given by the following form:
\begin{equation} \label{1}
\left(
           \begin{array}{cccccc}
\dagger & 0 & \cdots  &  \cdots  & \cdots & 0 \\
0 & \dagger & 0         & \cdots & 0 & \dagger \\
   & 0 & \ddots &           & \begin{sideways} $\ddots$ \end{sideways} & 0 \\
\vdots   & \vdots &    &    \dagger     &    &  \vdots      \\
   & 0 & \begin{sideways} $\ddots$ \end{sideways} & & \ddots & 0 \\ 
0  & \dagger & 0        & \cdots & 0 & \dagger \\
   \end{array}
         \right) \quad {\rm or} \quad
	\left(
           \begin{array}{ccccccc}
\dagger & 0 & \cdots  & \cdots & \cdots & \cdots & 0 \\
0 & \dagger & 0         & \cdots & \cdots & 0 & \dagger \\
   & 0 & \ddots &     &      & \begin{sideways} $\ddots$ \end{sideways} & 0 \\
\vdots   & \vdots &    &    \dagger & \dagger    &    &  \vdots      \\
\vdots   & \vdots &    &    \dagger & \dagger    &    &  \vdots      \\
   & 0 & \begin{sideways} $\ddots$ \end{sideways} & & & \ddots & 0 \\ 
0  & \dagger & 0   & \cdots     & \cdots & 0 & \dagger \\
   \end{array}
         \right),					
\end{equation}
when the size of the matrix is even or odd. 
Note that the locations of nonzero entries represented by ``$\dagger$'' in 
(\ref{1}) can appear only in the main diagonal and anti-lower-subdiagonal 
of $\breve{\mathbf F}_{\breve{\mu}} \breve{\mathbf{C}}
\breve{\mathbf F}^*_{\breve{\mu}}$. 
With such diagonalization and anti-lower-subdiagonalization structure,
we demonstrate that 
a quaternion circulant matrix can  
be block-diagonalized into 1-by-1 block and 2-by-2 block by permutated 
discrete quaternion Fourier matrix.
Therefore, the inverse of an invertible 
$n$-by-$n$ quaternion circulant matrix 
can be computed efficiently in $O(n \log n)$ operations similar
to the inverse of an invertible $n$-by-$n$ complex circulant matrix.
By using block-daigonalization results of quaternion circulant matrix,
we can derive quaternion tensor singular value 
decomposition of quaternion tensors where all entries are quaternion numbers.  

The outline of this paper is given as follows. 
In Section 2, we present block diagonalization of 
quaternion circulant matrices by discrete quaternion Fourier matrix. 
In Section 3, we study algebraic structure of 
quaternion tensor singular value decomposition of quaternion tensors.
In Section 4, numerical examples for quaternion inverse computation, and linear prediction of quaternion signals are tested to show the efficiency of our block diagonalization results. A numerical example of color video (third-order quaternion tensor) is employed to test the
effectiveness of quaternion tensor singular decomposition.
Finally, some concluding remarks are given in Section 5.

\subsection{Remarks}

Recently, a strategy based on octonion algebra to diagonalize quaternion circulant matrices was proposed in \cite{zheng2022block}. An ocotonion $\mathit{o}\in \mathbb{O}$ is often represented as $\mathit{o}=o_1+o_2 \mathtt{i} +o_3 \mathtt{j} +o_4 \mathtt{k} +o_5\mathtt{l}+o_6 \mathtt{il}+o_7 \mathtt{jl}+o_8  \mathtt{kl}$, where $o_t\in \mathbb{R}, (t=1,2,\cdots 8)$. And $1$, $ \mathtt{i},  \mathtt{j},  \mathtt{k},  \mathtt{l},  \mathtt{il},  \mathtt{jl},  \mathtt{kl}$ are known as unit octonions. Their multiplication rule is given in the following table.
\begin{table}[H]
  \centering 
  \begin{tabular}{c|cccccccc}
  $\mathbb{O}$& 1 & $\mathtt{i}$ & $\mathtt{j}$  & $\mathtt{k}$ &$\mathtt{l}$ &$\mathtt{il}$&$\mathtt{jl}$& $\mathtt{kl}$\\
   \hline
   1& 1 & $\mathtt{i}$ & $\mathtt{j}$  & $\mathtt{k}$ &$\mathtt{l}$ &$\mathtt{il}$&$\mathtt{jl}$& $\mathtt{kl}$\\
$ \mathtt{i}$& $ \mathtt{i}$ & -1 & $\mathtt{k}$  &$ -\mathtt{j}$ &$\mathtt{il}$ &$-\mathtt{l}$&$-\mathtt{kl}$& $\mathtt{jl}$\\
$ \mathtt{j}$& $ \mathtt{j}$ & -$\mathtt{k}$ & -1  &$ \mathtt{i}$ &$\mathtt{jl}$ &$\mathtt{kl}$&$-\mathtt{l}$& $-\mathtt{il}$\\
$ \mathtt{k}$& $ \mathtt{k}$ & $\mathtt{j}$ & -$ \mathtt{i}$  & -1&$\mathtt{kl}$ &-$\mathtt{jl}$&$\mathtt{il}$& $-\mathtt{l}$\\
$ \mathtt{l}$& $ \mathtt{l}$ & -$\mathtt{il}$ & -$ \mathtt{jl}$  & $-\mathtt{kl}$&-1 & $\mathtt{i}$&$\mathtt{j}$& $\mathtt{k}$\\
$ \mathtt{il}$& $ \mathtt{il}$ & $\mathtt{l}$ & -$ \mathtt{kl}$  & $\mathtt{jl}$&- $\mathtt{i}$&-1&-$\mathtt{k}$& $\mathtt{j}$\\

$ \mathtt{jl}$&$ \mathtt{jl}$& $ \mathtt{kl}$ & $\mathtt{l}$ & -$ \mathtt{il}$  & -$\mathtt{j}$&$\mathtt{k}$&-1&-$\mathtt{i}$ \\

$ \mathtt{kl}$&$ \mathtt{kl}$& $ -\mathtt{jl}$ & $\mathtt{il}$ & $ \mathtt{l}$  & -$\mathtt{k}$&-$\mathtt{j}$&$\mathtt{i}$&-1 \\
 \end{tabular}
 \caption{Octonion multiplication Rule}\label{table:Po_result}
\end{table}

In \cite{zheng2022block}, a unitary octonion matrix based on the unit octonion $\mathtt{l}$ (or $\mathtt{jl}$, or the linear combination of $\mathtt{l}$ and $\mathtt{jl}$: $(\mathtt{l}+\mathtt{jl})/\sqrt{2}$) is built to diagonalize a quaternion circulant matrix. Its diagonalization result can be achieved by the fast Fourier transform. The work seems to provide an optimistic prospect. However, 
the octonions do not satisfy the associative law. This leads to 
some computational issues.
  
\begin{example}
Given quaternion matrix $\breve{\mathbf{C}}$ and vector $\breve{\mathbf{b}}$, find the solution $\breve{\mathbf{x}}$ to  the following quaternion equation,
$$
\breve{\mathbf{C}}\breve{\mathbf{x}}=\breve{\mathbf{b}},
$$
where
$\breve{\mathbf{c}}=
\left(
           \begin{array}{c}
             -2+1 \mathtt{i}+1 \mathtt{j}+4 \mathtt{k} \\
             -1+2 \mathtt{i}+2 \mathtt{j}+3 \mathtt{k} \\
              1+3 \mathtt{i}+2 \mathtt{j}+2 \mathtt{k} \\
			  2+4\mathtt{i}+1 \mathtt{j}+1 \mathtt{k}
           \end{array}
         \right), \quad  \breve{\mathbf{b}}=\left(
           \begin{array}{c}
             -38+12 \mathtt{i}+19 \mathtt{j}+ 19\mathtt{k} \\
             -40+18 \mathtt{i}+17 \mathtt{j}+21\mathtt{k} \\
             -37+18 \mathtt{i}+18 \mathtt{j}+25 \mathtt{k} \\
			 -35+12\mathtt{i}+14\mathtt{j}+23 \mathtt{k}
           \end{array}
         \right).
$
The true solution is given by 
$$ 
\breve{\mathbf{x}}^*=\left(
           \begin{array}{c}
             2+2 \mathtt{i}+1 \mathtt{j}+2 \mathtt{k} \\
             2+1 \mathtt{i}+1 \mathtt{j}+1 \mathtt{k} \\
             2+2 \mathtt{i}+1 \mathtt{j}+1 \mathtt{k} \\
			 2+2\mathtt{i}+2\mathtt{j}+1 \mathtt{k}
           \end{array}
         \right).$$
According to \cite{zheng2022block}, there exists unitary octonion matrix $\mathbf{O}=\mathbf{F}_i\mathtt{l}$ such that $\mathbf{O}\breve{\mathbf{C}}\mathbf{O}^*=\breve{\mathbf{D}}$. Here $\breve{\mathbf{D}}\in \mathbb{Q}^{4\times 4}$ is a diagonal quaternion matrix. One can also choose  $\mathbf{O}=\mathbf{F}_i\mathtt{jl}$, or $\mathbf{O}=\mathbf{F}_i\mathtt{(l+jl)}/2$, which will lead to the same diagonal quaternion matrix $\breve{\mathbf{D}}$. Here $\breve{\mathbf{D}}$ and its inverse are given by
$$\breve{\mathbf{D}}=\Diag\left(
           \begin{array}{l}
 -10\mathtt{i}-6\mathtt{j}-10\mathtt{k} \\ 
 -1+5\mathtt{i}-\mathtt{j}- \mathtt{k} \\
  -2+2\mathtt{i}-2\mathtt{k}  \\
  -5-\mathtt{i}+3\mathtt{j}-3\mathtt{k} \end{array}
         \right),  \quad
   \breve{\mathbf{D}}^{-1}=\Diag\left( \begin{array}{l}
0.0424\mathtt{i}+0.0254\mathtt{j}+0.0424\mathtt{k}\\ 
  -0.0357-0.1786\mathtt{i}+0.0357\mathtt{j}+0.0357 \mathtt{k}\\
   -0.1667-0.1667\mathtt{i}+0.1667\mathtt{k}\\
  -0.1136+0.0227\mathtt{i}-0.0682\mathtt{j}+0.0682\mathtt{k}
  \end{array}
         \right),
 $$
where $\Diag(\breve{\mathbf{q}})$ returns to a diagonal matrix with the elements of vector $\breve{\mathbf{q}}$ on the main diagonal.
 Since 
$\mathbf{O}\mathbf{O}^*=\mathbf{O}^*\mathbf{O}=\mathbf{I}$, we can deduce that 
$$\breve{\mathbf{x}}_o=\breve{\mathbf{C}}^{-1}\breve{\mathbf{b}}=(\mathbf{O}^*\breve{\mathbf{D}}^{-1}\mathbf{O})\breve{\mathbf{b}}=\left(
           \begin{array}{l}
             2.2143+1.8571 \mathtt{i}+1.1169 \mathtt{j}+2.1753 \mathtt{k} \\
             1.8571+0.7857 \mathtt{i}+1.0390 \mathtt{j}+0.9740 \mathtt{k} \\
             1.7857+2.1429 \mathtt{i}+0.8831 \mathtt{j}+0.8247 \mathtt{k} \\
			 2.1429+2.2143\mathtt{i}+1.9610\mathtt{j}+1.0260 \mathtt{k}
           \end{array}
         \right).
$$
By using different calculation orders, we obtain 
$$
(\mathbf{O}^*\breve{\mathbf{D}}^{-1})(\mathbf{O}\breve{\mathbf{b}})=\left(
           \begin{array}{c}
             2.2143+0.8923 \mathtt{i}-0.0218 \mathtt{j}+1.9637 \mathtt{k} \\
             1.8571+0.9897 \mathtt{i}+1.1470 \mathtt{j}+ 1.9572\mathtt{k} \\
             1.7857+1.4637 \mathtt{i}+0.8354 \mathtt{j}+1.3923\mathtt{k} \\
			 2.1429+2.3663\mathtt{i}+0.6665\mathtt{j}+2.3987\mathtt{k}
           \end{array}
         \right),
$$
$$
\mathbf{O}^*(\breve{\mathbf{D}}^{-1}\mathbf{O}\breve{\mathbf{b}})=\left(
           \begin{array}{c}
             2.2695+1.0806 \mathtt{i}-0.4146 \mathtt{j}+1.8013 \mathtt{k} \\
             1.5617+0.8403 \mathtt{i}+0.8516\mathtt{j}+ 2.1065\mathtt{k} \\
             1.7305+1.2754\mathtt{i}+1.2282 \mathtt{j}+1.5546\mathtt{k} \\
			 2.4383+2.5156\mathtt{i}+0.9620\mathtt{j}+2.2494\mathtt{k}
           \end{array}
         \right),
$$
$$
\mathbf{O}^*(\breve{\mathbf{D}}^{-1}\mathbf{O})\breve{\mathbf{b}}=\left(
           \begin{array}{c}
             2.2695+1.6688 \mathtt{i}+1.3474 \mathtt{j}+1.8382 \mathtt{k} \\
             1.5617+0.9351\mathtt{i}+1.3669\mathtt{j}+ 0.9286\mathtt{k} \\
             1.7305+2.3312\mathtt{i}+0.6526\mathtt{j}+1.1818\mathtt{k} \\
			 2.4383+2.0649\mathtt{i}+1.6331\mathtt{j}+1.0714\mathtt{k}
           \end{array}
         \right),
$$
and 
$$
\mathbf{O}^*(\breve{\mathbf{D}}^{-1}(\mathbf{O}\breve{\mathbf{b}}))=\left(
           \begin{array}{c}
             2.2143+1.8571 \mathtt{i}+1.5714 \mathtt{j}+1.5000\mathtt{k} \\
             1.8571+0.7857\mathtt{i}+1.5000\mathtt{j}+ 1.5714\mathtt{k} \\
             1.7857+2.1429\mathtt{i}+0.4286\mathtt{j}+1.5000\mathtt{k} \\
			 2.1429+2.2143\mathtt{i}+1.5000\mathtt{j}+0.4286\mathtt{k}
           \end{array}
         \right).
$$
It is interesting to note that they are not the same, and it is clear  that they are not the 
solution of the above quaternion linear system. 
This is a limitation of the computational approach 
because of the non-associative nature in octonions.

In contrast, 
our proposed method is still based on the quaternion field.
Specifically, we diagonalize the circulant matrix by Algorithm \ref{algo:qft-circ} proposed in Section 2.2 with quaternion Fourier transform matrix $\breve{\mathbf{F}}_{\breve{\mu}}$, $\breve{\mu}=\frac{1}{\sqrt{3}}\mathtt{i}+\frac{1}{\sqrt{3}}\mathtt{j}+\frac{1}{\sqrt{3}}\mathtt{k}$. The resulting diagonal matix $\breve{\mathbf{\Lambda}}$ and its inverse are given as follows.
\begin{eqnarray*}
\breve{\mathbf{\Lambda}}=\left(
           \begin{array}{cccc}
        0   &      0     &    0   &      0\\
         0  & -2.4226     &    0  &       0 \\
         0   &      0  & -2.0000   &      0\\
         0    &     0    &     0  & -3.5774 \\     
           \end{array}
         \right)+ \left(
           \begin{array}{cccc}
            10.0000    &      0      &   0        & 0\\
         0   & 1.3987  &       0 &  -2.2440\\
         0    &     0  & -2.0000   &       0\\
         0 &  -1.0893  &       0   &-2.0654
            \end{array}
         \right)\mathtt{i} \\
         +\left(
           \begin{array}{cccc}
     6.0000   &       0   &      0   &      0\\
         0 &   1.3987   &      0  &  1.6427\\
         0  &       0  &0 &        0\\
         0   &-2.9761  &       0 &  -2.0654
            \end{array}
         \right)\mathtt{j}+\left(
           \begin{array}{cccc}
     
       10.0000   &    0  &       0  &       0\\
         0   & 1.3987   &      0  &  0.6013\\
         0    &     0  &  2.0000&        0\\
         0   & 4.0654  &       0   &-2.0654
            \end{array}
         \right)\mathtt{k},
\end{eqnarray*}
and 
\begin{eqnarray*}
\breve{\mathbf{\Lambda}}^{-1}=\left(
           \begin{array}{cccc}
        0   &      0     &    0   &      0\\
         0  &-0.1118    &    0  &       0 \\
         0   &      0  & -0.1667   &      0\\
         0    &     0    &     0  &-0.0757 \\     
           \end{array}
         \right)+ \left(
           \begin{array}{cccc}
           -0.0424    &      0      &   0        & 0\\
         0   & -0.0645  &       0 &  0.0188\\
         0    &     0  &0.1667   &       0\\
         0 & 0.1270   &       0   &0.0437
            \end{array}
         \right)\mathtt{i} \\
         +\left(
           \begin{array}{cccc}
    -0.0254   &       0   &      0   &      0\\
         0 &  -0.0645   &      0&   0.0513\\
         0  &       0  &0 &        0\\
         0   &-0.0930  &       0 &  0.0437
            \end{array}
         \right)\mathtt{j}+\left(
           \begin{array}{cccc}
        -0.0424   &    0  &       0  &       0\\
         0   & -0.0645   &      0  &  -0.0701\\
         0    &     0  & -0.1667&        0\\
         0   &-0.0340 &       0   & 0.0437
            \end{array}
         \right)\mathtt{k}.
\end{eqnarray*}
The solution of $\breve{\mathbf{C}}\breve{\mathbf{x}}=\breve{\mathbf{b}}$ is
 then computed by 
 $\breve{\mathbf{F}}^*_{\breve{\mu}}\breve{\mathbf{\Lambda}}^{-1}\breve{\mathbf{F}}_{\breve{\mu}}\breve{\mathbf{b}}$ that equals $\breve{\mathbf{x}}^*$.
\end{example}

\section{Quaternion Circulant Matrices}

\subsection{Preliminaries} 

In this subsection, we derive some useful results to characterize the 
transformation of a quaternion circulant matrix under $\mathtt{i}$, 
$\mathtt{j}$ and $\mathtt{k}$ to the other three orthogonal units.

\begin{defn}
We call two pure quaternions $ \breve{\mu}=\mu_1 \tt{i}+\mu_2 \tt{j}+ \mu_3 \tt{k}$ and $\breve{\alpha}=\alpha_1 \tt{i}+\alpha_2 \tt{j}+ \alpha_3 \tt{k}$ are orthogonal, denoted as 
$\breve{\mu} \perp \breve{\alpha} $, if $\breve{\mu}\cdot\breve{\alpha}=0$, that is
$
\mu_1 \alpha_1+\mu_2 \alpha_2+\mu_3 \alpha_3=0.
$
\end{defn}

Given any unit pure quaternion $\breve{\mu}=\mu_1 \tt{i}+\mu_2 \tt{j}+\mu_3 \tt{k}$, $|\breve{\mu}|=1$, and its orthogonal unit pure quaternion $\breve{\alpha}=\alpha_1 \tt{i}+\alpha_2 \tt{j}+\alpha_3 \tt{k}$, $|\breve{\alpha}|=1$. Their product $\breve{\beta}=\beta_1 \tt{i}+\beta_2 \tt{j}+\beta_3 \tt{k}$ is defined as $\breve{\beta}=\breve{\mu}\breve{\alpha}$, that is 
\begin{eqnarray*}
\breve{\beta}&=& (\mu_1 \tt{i}+\mu_2 \tt{j}+\mu_3 \tt{k})(\alpha_1 \tt{i}+\alpha_2 \tt{j}+\alpha_3 \tt{k})\\
&=&(\mu_2\alpha_3-\mu_3\alpha_2)\tt{i}+(\mu_3\alpha_1-\mu_1\alpha_3)\tt{j}+(\mu_1\alpha_2-\mu_2\alpha_1)\tt{k}.
\end{eqnarray*}

\begin{lemma}\label{lemm:pureQ}
For any pure quaternions $\breve{\mu}$ and $\breve{\alpha}$, if $\breve{\mu} \perp \breve{\alpha}$, then their product $\breve{\beta}=\breve{\mu}\breve{\alpha}$  satisfies
$$
\breve{\mu} \perp \breve{\beta}; \quad \breve{\alpha} \perp \breve{\beta}.
$$
\end{lemma}

\begin{proof}
The results follow by using $\breve{\alpha}\cdot\breve{\beta}=0$ and $\breve{\mu}\cdot\breve{\beta}=0$
\end{proof}

\begin{property}\label{prop:pureQ}
Given any unit pure quaternion $\breve{\mu}=\mu_1 \tt{i}+\mu_2 \tt{j}+\mu_3 \tt{k}$, and its orthogonal unit pure quaternion $\breve{\alpha}=\alpha_1 \tt{i}+\alpha_2 \tt{j}+\alpha_3 \tt{k}$, and their product $\breve{\beta}=\beta_1 \tt{i}+\beta_2 \tt{j}+\beta_3 \tt{k}$, then $(\breve{\mu},\breve{\alpha},\breve{\beta})$ satisfy
$$
\breve{\mu}^2=-1;\quad \breve{\alpha}^2=-1;\quad \breve{\beta}^2=-1; \quad \breve{\mu}\breve{\alpha}=-\breve{\alpha}\breve{\mu}=\breve{\beta};\quad \breve{\alpha}\breve{\beta}= -\breve{\beta}\breve{\alpha}=\breve{\mu};\quad 
\breve{\beta}\breve{\mu}=-\breve{\mu}\breve{\beta}=\breve{\alpha};
$$
\end{property}

\begin{proof}
$\breve{\beta}$ is the product of $\breve{\mu}$ and $\breve{\alpha}$, that is $\breve{\beta}=\breve{\mu}\breve{\alpha}$, 
$$
\breve{\alpha} \breve{\mu}=(\alpha_1 \tt{i}+\alpha_2 \tt{j}+\alpha_3 \tt{k})(\mu_1 \tt{i}+\mu_2 \tt{j}+\mu_3 \tt{k})=-\breve{\beta}.
$$
Also we have 
$$
\breve{\mu}^2=(\mu_1 \tt{i}+\mu_2 \tt{j}+\mu_3 \tt{k})^2=-(\mu^2_1+\mu^2_2+\mu^2_3)=-1;
$$
$$
 \breve{\alpha}^2=(\alpha_1 \tt{i}+\alpha_2 \tt{j}+\alpha_3 \tt{k})^2=-(\alpha^2_1+\alpha^2_2+\alpha^2_3)=-1. 
$$
$$
\breve{\beta}^2=\breve{\alpha} \breve{\mu}\breve{\alpha} \breve{\mu}=-\breve{\alpha} \breve{\mu} \breve{\mu}\breve{\alpha}=-1;\quad 
 \breve{\alpha}\breve{\beta}=\breve{\alpha}\breve{\alpha}\breve{\mu}=-\breve{\mu},
$$
and 
$$
\breve{\beta}\breve{\alpha}=\breve{\alpha}\breve{\mu}\breve{\alpha}=-\breve{\alpha}\breve{\alpha}\breve{\mu}=\breve{\mu}. 
$$
Similarly,
$\breve{\beta}\breve{\mu}=-\breve{\mu}\breve{\beta}=\breve{\alpha}$. The results follow.
\end{proof}

From Property \ref{prop:pureQ}, $(\breve{\mu},\breve{\alpha},\breve{\beta})$ has the same properties as $(\tt{i},\tt{j},\tt{k})$. Hence  $(\breve{\mu},\breve{\alpha},\breve{\beta})$ can be regarded as three-axis system 
like $(\tt{i},\tt{j},\tt{k})$. In the following, we will show that any quaternion matrix can be 
rewritten in the three-axis system: $(\breve{\mu},\breve{\alpha},\breve{\beta})$.

\begin{lemma}\label{lemm:Minmu}
Given a quaternion matrix $\breve{\mathbf{M}}=\mathbf{M}_0+\mathbf{M}_1 \tt{i}+\mathbf{M}_2 \tt{j}+\mathbf{M}_3 \tt{k}\in \mathbb{Q}^{m\times n}$, where $\mathbf{M}_l\in \mathbb{R}^{m\times n}$, $l=0,1,2,3$, and unit pure quaternion three-axis system $(\breve{\mu},\breve{\alpha},\breve{\beta})$, then 
$$
\breve{\mathbf{M}}=\mathbf{A}_0+\mathbf{A}_1 \breve{\mu}+\mathbf{A}_2 \breve{\alpha}+\mathbf{A}_3 \breve{\beta},$$
where $\mathbf{A}_0=\mathbf{M}_0$, $\mathbf{A}_1=\sum\limits^3_{l=1}\mu_l \mathbf{M}_l$, $\mathbf{A}_2=\sum\limits^3_{l=1}\alpha_l \mathbf{M}_l$, $\mathbf{A}_3=\sum\limits^3_{l=1}\beta_l \mathbf{M}_l$.
\end{lemma}

\begin{proof}
\begin{eqnarray*}
\breve{\mathbf{M}}&=&\mathbf{A}_0+\mathbf{A}_1 \breve{\mu}+\mathbf{A}_2 \breve{\alpha}+\mathbf{A}_3 \breve{\beta}\\
&=&\mathbf{A}_0+\mathbf{A}_1 (\mu_1 \tt{i}+\mu_2 \tt{j}+\mu_3 \tt{k})+\mathbf{A}_2 (\alpha_1 \tt{i}+\alpha_2 \tt{j}+\alpha_3 \tt{k})+\mathbf{A}_3 (\beta_1 \tt{i}+\beta_2 \tt{j}+\beta_3 \tt{k})\\
&=&\mathbf{A}_0+(\mu_1 \mathbf{A}_1+\alpha_1 \mathbf{A}_2 +\beta_1 \mathbf{A}_3)\tt{i}+(\mu_2 \mathbf{A}_1+\alpha_2\mathbf{A}_2 +\beta_2 \mathbf{A}_3)\tt{j}
+(\mu_3 \mathbf{A}_1+\alpha_3\mathbf{A}_2 +\beta_3 \mathbf{A}_3)\tt{k}.
\end{eqnarray*}
Since $\breve{\mathbf{M}}=\mathbf{M}_0+\mathbf{M}_1 \tt{i}+\mathbf{M}_2 \tt{j}+\mathbf{M}_3 \tt{k}$, from Lemma \ref{lemm:pureQ}, we deduce that
$$
\left(
           \begin{array}{c}
             \mathbf{M}_1  \\
           \mathbf{M}_2\\
           \mathbf{M}_3
           \end{array}
         \right)=\left(
           \begin{array}{ccc}
            \mu_1 & \alpha_1& \beta_1\\
            \mu_2 & \alpha_2& \beta_2\\
             \mu_3 & \alpha_3& \beta_3
           \end{array}
         \right)\left(
           \begin{array}{c}
             \mathbf{A}_1  \\
           \mathbf{A}_2\\
           \mathbf{A}_3
           \end{array}
         \right),\quad \text{then}
 \left(
           \begin{array}{c}
             \mathbf{A}_1  \\
           \mathbf{A}_2\\
           \mathbf{A}_3
           \end{array}
         \right)=\left(
           \begin{array}{ccc}
            \mu_1 & \alpha_1& \beta_1\\
            \mu_2 & \alpha_2& \beta_2\\
             \mu_3 & \alpha_3& \beta_3
           \end{array}
         \right)^T\left(
           \begin{array}{c}
             \mathbf{M}_1  \\
           \mathbf{M}_2\\
           \mathbf{M}_3
           \end{array}
         \right).
$$
The results follow. 
\end{proof}

According to Lemma 2, to obtain the new representation
in three-axis system $(\breve{\mu},\breve{\alpha},\breve{\beta})$, we can combine the matrices in the original
three-axis system $(\tt{i},\tt{j},\tt{k})$. The following lemma presents a property of quaternions $\breve{\mathbf{X}}$ which can be expressed as $\mathbf{X}_0+\mathbf{X}_1 \breve{\mu}.$

\begin{lemma}\label{lem:qua-commu}
Given unit pure quaternion three-axis system 
$(\breve{\mu},\breve{\alpha},\breve{\beta})$, for any $\breve{\mathbf{X}}=\mathbf{X}_0+\mathbf{X}_1 \breve{\mu} \in \mathbb{Q}^{m\times n}$,
$$
\breve{\mu} \breve{\mathbf{X}}=\breve{\mathbf{X}}\breve{\mu},\quad
\breve{\alpha} \breve{\mathbf{X}}=(\breve{\mathbf{X}}^*)^T\breve{\alpha},\quad
\breve{\beta} \breve{\mathbf{X}}=(\breve{\mathbf{X}}^*)^T\breve{\beta},
$$
\end{lemma}

\begin{proof} From Property \ref{prop:pureQ}, we have
$$
\breve{\mu}\breve{\mathbf{X}}=\mathbf{X}_0\breve{\mu}+\mathbf{X}_1\breve{\mu}^2=\mathbf{X}_0\breve{\mu}-\mathbf{X}_1=\breve{\mathbf{X}}\breve{\mu};
$$
$$
\breve{\alpha}\breve{\mathbf{X}}=\mathbf{X}_0\breve{\alpha}+\mathbf{X}_1\breve{\alpha}\breve{\mu}=\mathbf{X}_0\breve{\alpha}-\mathbf{X}_1\breve{\beta}=(\breve{\mathbf{X}}^*)^T\breve{\alpha};
$$
$$
\breve{\beta}\breve{\mathbf{X}}=\mathbf{X}_0\breve{\beta}+\mathbf{X}_1\breve{\beta}\breve{\mu}=\mathbf{X}_0\breve{\beta}+\mathbf{X}_1\breve{\alpha}=(\breve{\mathbf{X}}^*)^T\breve{\beta}.
$$
The results follow.
\end{proof}

\subsection{Block-Diagonalization}

Given a quaternion circulant matrix $\breve{\mathbf{C}}\in \mathbb{Q}^{n\times n}$,
$$
\breve{\mathbf{C}}=\mathbf{C}_0+\mathbf{C}_1\mathtt{i}+\mathbf{C}_2\mathtt{j}+\mathbf{C}_3\mathtt{k}
$$
where $\mathbf{C}_l$, $l=0,1,2,3$ are circulant real matrices.


\begin{lemma}
Given unit pure quaternion three-axis system $(\breve{\mu},\breve{\alpha},\breve{\beta})$, any quaternion circulant matrix $\breve{\mathbf{C}}$ can be represented in $(\breve{\mu},\breve{\alpha},\breve{\beta})$, that is
$$
\breve{\mathbf{C}}=\mathbf{S}_0+\mathbf{S}_1\breve{\mu}+\mathbf{S}_2\breve{\alpha}+\mathbf{S}_3\breve{\beta},
$$
where $\mathbf{S}_l\in \mathbb{R}^{n\times n}$, $l=0,1,2,3$.
Then the coefficient matrices $\{\mathbf{S}_l \}^3_{l=0}$ are also circulant matrices. 
\end{lemma}

\begin{proof}
The result follows directly by Lemma \ref{lemm:Minmu}.
\end{proof}

\begin{lemma}\label{prop:QFM}
Let $\breve{\mathbf{F}}_{\breve{\mu}}$ be 
the quaternion discrete Fourier matrix, defined in 
(\ref{qfft}), then
$$
\breve{\mathbf{F}}^*_{\breve{\mu}}
\breve{\mathbf{F}}_{\breve{\mu}}=
\breve{\mathbf{F}}_{\breve{\mu}}
\breve{\mathbf{F}}^*_{\breve{\mu}}=\mathbf{I}_n;
\quad 
\breve{\mathbf{F}}^2_{\breve{\mu}}=(
\breve{\mathbf{F}}^*_{\breve{\mu}})^2=\mathbf{A}; 
$$
where 
\begin{equation}\label{P}
\mathbf{A} = \left(
           \begin{array}{ccccc}
            1 & 0& \cdots& 0&0\\
            0& 0&\cdots&0&1\\
             0 &0& \cdots&1&0\\
             \vdots&\vdots&\begin{sideways} $\ddots$ \end{sideways}
             &\vdots&\vdots\\
             0&1&\cdots&0&0
           \end{array}
         \right).
 \end{equation}
\end{lemma}

\begin{proof} Let $\breve{\omega}=\exp \left ( {\displaystyle 
\frac{-2 \pi \breve{\mu}}{n}} \right )$, then
$$
(\breve{\mathbf{F}}^*_{\mu}
\breve{\mathbf{F}}_{\mu})_{st}=
\frac{1}{n}
\left (1+\breve{\omega}^{t-s}+\breve{\omega}^{2(t-s)}+\cdots+\breve{\omega}^{(n-1)(t-s)} \right),
$$
when $s=t$, $ [ 
\breve{
\mathbf{F}}^*_{\breve{\mu}}
\breve{
\mathbf{F}}_{\breve{\mu}}] _{st}=1$; when $s\neq t$, 
$[
\breve{
\mathbf{F}}^*_{\breve{\mu}}
\breve{
\mathbf{F}}_{\breve{\mu}}]_{st}=
{\displaystyle 
\frac{1}{n} \cdot \frac{(\breve{\omega}^{t-s})^n-1}{\breve{\omega}^{t-s}-1}}=0$. 
That is $
\breve{
\mathbf{F}}^*_{\breve{\mu}}
\breve{
\mathbf{F}}_{\breve{\mu}}=\mathbf{I}_n$. 
Similarly, $\breve{\mathbf{F}}_{\breve{\mu}}
\breve{\mathbf{F}}^*_{\breve{\mu}}=\mathbf{I}_n$.
$$
[\breve{\mathbf{F}}^2_{\breve{\mu}}]_{st}=\frac{1}{N}(1+\breve{\omega}^{t+s-2}+\breve{\omega}^{2(t+s-2)}+\cdots+\breve{\omega}^{(n-1)(t+s-2)}),
$$
when $s=t=1$ or $t+s-2=n$, 
$[
\breve{\mathbf{F}}^2_{\breve{\mu}}]_{st}=1$; for other entries, $[
\breve{\mathbf{F}}^2_{\breve{\mu}}]_{st}=0$. 
It implies that $\breve{\mathbf{F}}^2_{\breve{\mu}}=\mathbf{A}$. 
Similarly, $(\breve{\mathbf{F}}^*_{\breve{\mu}})^2=\mathbf{A}$. 
The results follow. 
\end{proof}

\begin{lemma}\label{lem:6}
For any real circulant matrix $\mathbf{C}=\text{circ}(\mathbf{c})\in \mathbb{R}^{n\times n}$, with $\mathbf{c}=(c^{(0)},c^{(1)},\cdots,c^{(n-1)})$, given 
discrete quaternion Fourier transform matrix 
$\breve{\mathbf{F}}_{\breve{\mu}}$, then 
$\breve{\mathbf{F}}_{\breve{\mu}}\mathbf{C} \breve{\mathbf{F}}^*_{\breve{\mu}}= \Diag (\sqrt{n}\mathbf{F}_{\breve{\mu}}\mathbf{c})$, where $Diag(\mathbf{c})$ returns a square diagonal matrix with the elements of vector $\mathbf{c}$ on the main diagonal.
\end{lemma}

\begin{proof}
It is well known that any real circulant matrix $\mathbf{C}=\text{circ}
(c^{(0)},c^{(1)},\cdots,c^{(n-1)})$ can be represented as 
\begin{equation}
\mathbf{C}=c^{(0)}\mathbf{I}_n+c^{(1)}\mathbf{J}+c^{(2)}(\mathbf{J})^2+
\cdots+c^{(n-1)}(\mathbf{J})^{n-1},
\end{equation}
where 
\begin{equation}\label{J}
\mathbf{J}=
  \left(
           \begin{array}{ccccc}
            0 & 0& \cdots& 0&1\\
            1& 0&\cdots&0&0\\
             0 &1& \cdots&0&0\\
             \vdots&\vdots&\ddots&\vdots&\vdots\\
             0&0&\cdots&1&0
           \end{array}
         \right).
\end{equation}   
Note ${\mathbf J}$ is also a circulant matrix with $\text{circ}
(0,1,0,\cdots,0)$.
Let $\breve{\lambda}$ be the eigenvalues of $\mathbf{J}$, then 
$c^{(0)}+c^{(1)} \breve{\lambda}+c^{(2)} \breve{\lambda}^2+\cdots+c^{(n-1)}
\breve{\lambda}^{n-1}$ are also the eigenvalues of $\mathbf{C}$.
From ${\rm det}(\breve{\lambda} \mathbf{I}_n-\mathbf{J})=0$, we have 
$\breve{\lambda}^n-1=0$.
It implies that $\breve{\lambda}=1,
\breve{\omega}, \breve{\omega}^2,\cdots, \breve{\omega}^{n-1}$ where  
$\breve{\omega}=\exp \left ( {\displaystyle \frac{-2 \pi \breve{\mu}}{ n}} \right )$ and $\breve{\mu}$ is a unit pure quaternion.
We can check 
$\breve{\mathbf{F}}^*_{\breve{\mu}}$ are eigenvector matrix of 
$\mathbf{C}$. It is easy to deduce that $\breve{\mathbf{F}}_{\breve{\mu}} 
\breve{\mathbf{C}} \breve{\mathbf{F}}^*_{\breve{\mu}}= \Diag(\sqrt{n}
\breve{\mathbf{F}}_{\breve{\mu}} \mathbf{c})$. The results follow.
\end{proof}

It is well-known that $\breve{\mathbf{F}}^*_{\tt{i}}$ is the
eigenvector matrix of $\mathbf{C}$. In Lemma \ref{lem:6}, we indeed demonstrate 
that $\breve{\mathbf{F}}^*_{\breve{\mu}}$ are also eigenvector 
matrix of $\mathbf{C}$ when $\breve{\mu}$ is a unit pure quaternion.

\begin{theorem}\label{Thm:Sec2-1}
Given a quaternion circulant matrix in unit pure quaternion 
three-axis system $(\breve{\mu},\breve{\alpha},\breve{\beta})$, 
$$
\breve{\mathbf{S}}=\mathbf{S}_0+\mathbf{S}_1\breve{\mu}+\mathbf{S}_2\breve{\alpha}+\mathbf{S}_3\breve{\beta}.
$$
Then $\mathbf{P} \breve{\mathbf{F}}_{\breve{\mu}}$ can block-diagonalize $\breve{\mathbf{S}}$, where
$\mathbf{P}$ is 
a permutation matrix $\mathbf{P}$ 
by the exchanging of $k$-th and $(n-k+3)$-th rows for $k=3,4,\ldots,\frac{n+1}{2}$
when $n$ is odd; 
by the exchanging of $k$-th and $(n-k+3)$-th rows for $k=3,4,\ldots,\frac{n}{2}$
when $n$ is even. 
More precisely, for odd $n$,
$\mathbf{P} \breve{\mathbf{F}}_{\breve{\mu}}\breve{\mathbf{S}}
\breve{\mathbf{F}}^*_{\breve{\mu}} \mathbf{P}^*$ is a diagonal block matrix
where the first diagonal block is 1-by-1 matrix, and 
the other diagonal blocks are 2-by-2 matrices; for even $n$, except for the above diagonal blocks, $\mathbf{P} \breve{\mathbf{F}}_{\breve{\mu}}\breve{\mathbf{S}}
\breve{\mathbf{F}}^*_{\breve{\mu}} \mathbf{P}^*$ has one more 1-by-1 matrix located in the last diagonal block. 
\end{theorem}

\begin{proof}
Note that 
$\{\mathbf{S}_l\}^3_{l=0}$ are real circulant matrices. Let $\mathbf{S}_l=\text{circ}(\mathbf{s}_l)$, $l=0,1,2,3$. Then from Lemma \ref{lem:qua-commu} and Lemma \ref{prop:QFM}, we have
\begin{eqnarray*}
\breve{\mathbf{F}}_{\breve{\mu}}\breve{\mathbf{S}}
\breve{\mathbf{F}}^*_{\breve{\mu}}&=& 
\breve{\mathbf{F}}_{\breve{\mu}}(\mathbf{S}_0+\mathbf{S}_1\breve{\mu}+\mathbf{S}_2\breve{\alpha}+\mathbf{S}_3\breve{\beta})
\breve{\mathbf{F}}^*_{\breve{\mu}}\\
&=& 
\breve{\mathbf{F}}_{\breve{\mu}}\mathbf{S}_0 \breve{\mathbf{F}}^*_{\breve{\mu}}+
\breve{\mathbf{F}}_{\breve{\mu}}\mathbf{S}_1 \breve{\mu}
\breve{\mathbf{F}}^*_{\breve{\mu}}+ 
\breve{\mathbf{F}}_{\breve{\mu}}\mathbf{S}_2\breve{\alpha}
\breve{\mathbf{F}}^*_{\breve{\mu}}+ 
\breve{\mathbf{F}}_{\breve{\mu}}\mathbf{S}_3\breve{\beta}
\breve{\mathbf{F}}^*_{\breve{\mu}}\\
&=& 
\breve{\mathbf{F}}_{\breve{\mu}}\mathbf{S}_0
\breve{\mathbf{F}}^*_{\breve{\mu}}+
( \breve{\mathbf{F}}_{\breve{\mu}}\mathbf{S}_1
\breve{\mathbf{F}}^*_{\breve{\mu}})
\breve{\mathbf{F}}_{\breve{\mu}}\breve{\mu}
\breve{\mathbf{F}}^*_{\breve{\mu}}+ (
\breve{\mathbf{F}}_{\breve{\mu}}\mathbf{S}_2
\breve{\mathbf{F}}^*_{\breve{\mu}})
\breve{\mathbf{F}}_{\breve{\mu}}
\breve{\alpha}
\breve{\mathbf{F}}^*_{\breve{\mu}}+ (
\breve{\mathbf{F}}_{\breve{\mu}}\mathbf{S}_3
\breve{\mathbf{F}}^*_{\breve{\mu}})
\breve{\mathbf{F}}_{\breve{\mu}}\breve{\beta}
\breve{\mathbf{F}}^*_{\breve{\mu}}\\
&=&\breve{\mathbf{\Lambda}}_0+
\breve{\mathbf{\Lambda}}_1 \breve{\mu}+
\breve{\mathbf{\Lambda}}_2 
\breve{\mathbf{F}}^2_{\breve{\mu}}\breve{\alpha}+
\breve{\mathbf{\Lambda}}_3
\breve{\mathbf{F}}^2_{\breve{\mu}}\breve{\beta}\\
&=&
\breve{\mathbf{\Lambda}}_0+
\breve{\mathbf{\Lambda}}_1\breve{\mu}+
\breve{\mathbf{\Lambda}}_2\mathbf{A}\breve{\alpha}+
\breve{\mathbf{\Lambda}}_3\mathbf{A}\breve{\beta}\doteq \breve{\mathbf{\Lambda}},
\end{eqnarray*}
where $
\breve{\mathbf{\Lambda}}_l=\Diag(\sqrt{n}
\breve{\mathbf{F}}_{\breve{\mu}} \mathbf{s}_l)$, $l=0,1,2,3$.
It is clear that the non-zero patterns of  
$\breve{\mathbf{\Lambda}}_2\mathbf{A}$
and 
$\breve{\mathbf{\Lambda}}_3\mathbf{A}$
are the same as that of $\mathbf{A}$.
Note that the nonzero entries of
$\breve{\mathbf{\Lambda}}$ can only appear 
in the main diagonal and anti-lower-subdiagonal locations. 
The structure of $\breve{\mathbf{\Lambda}}$ is given in (\ref{1}).
According to such structure, we can swap the 
$k$-th and $(n-k+3)$-th rows and columns for $k=3,4,\ldots,\frac{n}{2}$
when $n$ is even; or  
swap the 
$k$-th and $(n-k+3)$-th rows and columns for $k=3,4,\ldots,\frac{n+1}{2}$
when $n$ is odd. In this way, we obtain 
1-by-1 block in the first main diagonal position,
and 2-by-2 block in other diagonal positions when $n$ is odd; and for even $n$, except the above blocks, there is one more 1-by-1 matrix located in the last diagonal block. 
Equivalently, we just respectively apply $\mathbf{P}$ and $\mathbf{P}^*$ on 
the left-hand and the right-hand sides of 
$\breve{\mathbf{\Lambda}}_0+
\breve{\mathbf{\Lambda}}_1\breve{\mu}+
\breve{\mathbf{\Lambda}}_2\mathbf{A}\breve{\alpha}+
\breve{\mathbf{\Lambda}}_3\mathbf{A}\breve{\beta}$.
\end{proof}

\begin{remark}
We remark that $(\tt{i},\tt{j},\tt{k})$ are unit pure quaternion three-axis system. For any quaternion circulant matrix $\breve{\mathbf{C}}$ 
represented in $(\tt{i},\tt{j},\tt{k})$, the standard discrete Fourier transform matrix $\breve{\mathbf{F}}_{\tt i}$ can block-diagonalize $\breve{\mathbf{C}}$ due to Theorem \ref{Thm:Sec2-1}. 
\end{remark}

\begin{remark}
Theorem \ref{Thm:Sec2-1} tells us that a quaternion circulant matrix in unit pure quaternion axis $(\breve{\mu},\breve{\alpha},\breve{\beta})$ can be block-diagonalized by $\breve{\mathbf{F}}_{\breve{\mu}}$. However, given $\breve{\mu}$, one can have many choices for $(\breve{\alpha},\breve{\beta})$ to form unit pure quaternion three-axis system. Note that 
$\breve{\mathbf{F}}_{\breve{\mu}}\breve{\mathbf{S}}\breve{\mathbf{F}}^*_{\breve{\mu}}= \breve{\mathbf{F}}_{\breve{\mu}}\breve{\mathbf{C}}
\breve{\mathbf{F}}^*_{\breve{\mu}}$ is 
valid for $\breve{\mathbf{C}}$ that is represented in three-axis system 
$(\tt{i},\tt{j},\tt{k})$. It is easy to derive that
when $\breve{\mu}$ is fixed, different choices of $(\breve{\alpha},\breve{\beta})$ would not change the block-diagonalized structure of $\breve{\mathbf{\Lambda}}$, though the values of $(\breve{\mathbf{\Lambda}}_2,\breve{\mathbf{\Lambda}}_3)$ would be different.
\end{remark}

\begin{remark}
The trivial corollary is that $\breve{\mathbf{C}}$ without entries in 
$\tt{j}$ and $\tt{k}$ can be diagonalized by $\breve{\mathbf{F}}_{\tt i}$. 
\end{remark}
 
The block-diagonalization procedure of a quaternion circulant matrix is presented in Algorithm 1. The cost of computing 
discrete quaternion Fourier matrix on a $n$-vector 
is of $O(n \log n)$ operations,
see \cite{Pei}, and the computational complexity of the 
whole block-diagonalization of 
an $n$-by-$n$ quaternion circulant matrix 
is of $O(n \log n)$ operations. 
When $\breve{\mathbf{C}}$ is invertible, and the cost 
of computing $\breve{\mathbf{C}}^{-1} \breve{\mathbf{x}}$ is also 
of $O(n \log n)$ operations. We just note that 
$$
\breve{\mathbf{C}}^{-1} \breve{\mathbf{x}} = 
\breve{\mathbf{F}}_{\breve{\mu}}^* 
\mathbf{P}^* \breve{\mathbf{\Lambda}}^{-1} 
 \mathbf{P}\breve{\mathbf{F}}_{\breve{\mu}} \breve{\mathbf{x}},
$$
and therefore the computation involves block-diagonalization of 
$\breve{\mathbf{C}}$, quaternion FFTs and the inverse of 2-by-2 matrices.

\begin{algorithm}
\caption{Block-diagonalization of Quaternion Circulant Matrix \label{algo:qft-circ}}
\begin{algorithmic}[1]
\REQUIRE Given quaternion circulant matrix $\breve{\mathbf{C}}= \mathbf{C}_0+\mathbf{C}_1 \tt{i}+\mathbf{C}_2 \tt{j}+\mathbf{C}_3 \tt{k} \in \mathbb{Q}^{n\times n}$;  quaternion Fourier transform matrix $\breve{\mathbf{F}}_{\breve{\mu}}$, 

\ENSURE Block-diagonalization of quaternion circulant matrix $\breve{\mathbf{\Lambda}}$.
 
\STATE   Generate unit pure quaternion three-axis system $(\breve{\mu},\breve{\alpha},\breve{\beta})$.

\STATE Rewrite quaternion circulant matrix $\breve{\mathbf{C}}$ in  axis $(\breve{\mu},\breve{\alpha},\breve{\beta})$.\\
$\mathbf{S}_0=\mathbf{C}_0$;\\
$\mathbf{S}_1=\mu_1\mathbf{C}_1+\mu_2\mathbf{C}_2+\mu_3\mathbf{C}_3$;\\
$\mathbf{S}_2=\alpha_1\mathbf{C}_1+\alpha_2\mathbf{C}_2+\alpha_3\mathbf{C}_3$;\\
$\mathbf{S}_3=\beta_1\mathbf{C}_1+\beta_2\mathbf{C}_2+\beta_3\mathbf{C}_3$; 
 
\STATE QFFT on $\breve{\mathbf{S}}= \mathbf{S}_0+\mathbf{S}_1 \breve{\mu}+\mathbf{S}_2 \breve{\alpha}+\mathbf{S}_3 \breve{\beta}$.\\
$\breve{\mathbf{\Lambda}}_0=\Diag(\sqrt{n}\breve{\mathbf{F}}_{\breve{\mu}}\mathbf{s}_0); $ 
$\quad \breve{\mathbf{\Lambda}}_1=\Diag(\sqrt{n}\breve{\mathbf{F}}_{\breve{\mu}}\mathbf{s}_1)$; \\ 
$\breve{\mathbf{\Lambda}}_2=\Diag(\sqrt{n}\breve{\mathbf{F}}_{\breve{\mu}}\mathbf{s}_2)$;  
$\quad \breve{\mathbf{\Lambda}}_3=\Diag(\sqrt{n}\breve{\mathbf{F}}_{\breve{\mu}}\mathbf{s}_3)$.

\STATE  $\breve{\mathbf{\Lambda}}=\breve{\mathbf{\Lambda}}_0+\breve{\mathbf{\Lambda}}_1\breve{\mu}+\breve{\mathbf{\Lambda}}_2\mathbf{A}\breve{\alpha}+\breve{\mathbf{\Lambda}}_3\mathbf{A}\breve{\beta}$. 
 
\end{algorithmic}
\end{algorithm}
We remark that Step 3 in Algorithm \ref{algo:qft-circ} can be simply implemented by using MATLAB toolbox QTFM \footnote{https://qtfm.sourceforge.io/}. Precisely, 
$$
\breve{\bm{\lambda}}_l=\text{qfft}(\mathbf{s}_l,\breve{\mu}, \text{'L'}),\quad l=0,1,2,3,
$$
where $\breve{\bm{\lambda}}_l$ is the main diagonal of $\breve{\mathbf{\Lambda}}_l$.

In the next section, we will make use of the results to study quaternion
tensor singular value decomposition and its algebraic structure.
In Section 4, we will demonstrate our proposed block diagonalization results are 
efficient in linear prediction of quaternion signal processing.

\section{Quaternion Tensor Singular Value Decomposition}

In this section, we will apply the block-diagonalization results of quaternion circulant matrices to study quaternion tensor singular value decomposition
so that we can use the decomposition for color videos which are represented as third-order quaternion tensors.

\subsection{Tensor Singular Value Decomposition}

Firstly we will briefly review the background and introduce the operations used in tensor decomposition.  To exploit the inherent structure of tensors, Kilmer et al. first studied an operator named tensor-tensor product ($t$-product) in \cite{kilmer2011factorization}. The operator built based on the complex 
Fourier transform gives a new interpretation of the complex third-order tensors on the oriented matrix space. 
Tensor Singular Value Decomposition (T-SVD) and its associated rank called tubal rank were then proposed based on the $t$-product. The new decomposition can well characterize the inherent low-rank structure of complex third-order tensors.
For real/complex tensors, researchers
\cite{kernfeld2015tensor,martin2013order,song2019robust} studied T-SVD based on cosine transform, and other variants based on the transforms which are invertible. As expected, due to its characteristics, T-SVD shows great advantages in capturing spatial-shift correlations in real-world data, especially in image deblurring and completion problems 
  \cite{ kilmer2013third,martin2013order,zhang2014novel, zhang2016exact,
jiang2019robust,song2019robust,zhang2021low,zhou2017tensor}.

In the following, we will review some basic definitions from \cite{kilmer2011factorization,kilmer2013third} for the complex T-SVD.

\begin{defn}[\cite{kilmer2011factorization,kilmer2013third}]
Given a complex tensor $\mathcal{T}$ of $n_1\times n_2\times m$, then $\bcirc(\mathcal{T})$
and $\unfold(\mathcal{T})$ are defined respectively as follows:
$$
\bcirc(\mathcal{T})= 
  \left(
           \begin{array}{ccccc}
            \mathbf{T}_1 & \mathbf{T}_m & \cdots &\mathbf{T}_3 &\mathbf{T}_2 \\
            \mathbf{T}_2 & \mathbf{T}_1 &\cdots &  &\mathbf{T}_3\\
           \vdots&\ddots&\ddots&\ddots&\vdots\\
      \mathbf{T}_{m-1} & &\cdots&\mathbf{T}_1&\mathbf{T}_m \\
						\mathbf{T}_m&\mathbf{T}_{m-1}&\cdots&\mathbf{T}_2&\mathbf{T}_1
           \end{array}
         \right),\quad \unfold(\mathcal{T})= \left(
           \begin{array}{c}
           \mathbf{T}_1 \\
            \mathbf{T}_2\\
           \vdots\\
            \mathbf{T}_m 
           \end{array}
         \right);            
$$
where $\{\mathbf{T}_t\}^m_{t=1}$ are frontal slices of tensor, i.e., $\mathbf{T}_t=\mathcal{T}(:,:,t)$ for $t=1,2,\cdots,m$. We use $``\fold"$ to return frontal slices $\{\mathbf{T}_t\}^m_{t=1}$ to tensor $\mathcal{T}$, precisely, $\fold(\unfold(\mathcal{T}))=\mathcal{T}$. $``\Bdiag"$ returns a block-diagonal matrix with the frontal slices of tensor $\mathcal{T}$ in the diagonal, specifically,
$$
\Bdiag(\mathcal{T})\doteq \Diag(\mathbf{T}_1,
\cdots,\mathbf{T}_m)
= \left ( 
\begin{array}{cccc}
 \mathbf{T}_1 & & & 0\\
          &   \mathbf{T}_2 & & \\
        & &   \ddots & \\
       0 & & &      \mathbf{T}_m \\
\end{array}
\right ).
$$
\end{defn}

\begin{defn}[t-product\cite{kilmer2011factorization,kilmer2013third}]
Let $\mathcal{T}$ be a complex tensor of $n_1\times n_2\times m $ and 
$\mathcal{B}$ be a complex tensor of $n_2\times n_3\times m$, then the $t$-product $\mathcal{T}\star \mathcal{B}$ is a complex tensor of $n_1\times n_3 \times m$, that is
\begin{equation}
\mathcal{T}\star \mathcal{B} = \fold(\bcirc(\mathcal{T}) \unfold(\mathcal{B})).
\end{equation}
\end{defn}

 Note that the $t$-product in the real or complex case 
can be described as the multiplication 
of a block circulant matrix and a block matrix column. In this section,
we are interested in the case of quaternion numbers, it is valid to 
describe the $t$-product as the multiplication of a block
quaternion circulant matrix and a block quaternion matrix column by
using the proposed block diagonalization results.
It is equivalent to applying the quaternion Fourier transform into the tubes of the third mode
of a quaternion tensor.

\begin{defn}[Identity tensor\cite{kilmer2011factorization,kilmer2013third}]
We call a tensor $\mathcal{I}$ of $n\times n\times m$ identity tensor if its first frontal slice is $n\times n$ identity matrix, and other frontal slices are all zeros.
\end{defn}

\begin{defn}[Conjugate transpose\cite{kilmer2011factorization,kilmer2013third}]
Given a complex tensor $\mathcal{B}$ of $n_1\times n_2\times m$, then its conjugate transpose $\mathcal{B}^*$ is tensor of $n_2\times n_1\times m$ obtained by conjugate transposing each of the frontal slices and then reversing the order of transposed frontal slices from $2$ through $m$.
\end{defn}

\begin{defn}[Orthogonal tensor\cite{kilmer2011factorization,kilmer2013third}]
A tensor $\mathcal{Q}$ of $n\times n\times m$ is unitary if $\mathcal{Q}^*\star \mathcal{Q}=\mathcal{Q}\star \mathcal{Q}^*=\mathcal{I}$.
\end{defn}

\begin{theorem}[T-SVD\cite{kilmer2011factorization,kilmer2013third}]
Given a tensor $\mathcal{T}\in \mathbb{C}^{n_1\times n_2\times m}$, then $\mathcal{T}$ can be factorized as
\begin{equation}
\mathcal{T}=\mathcal{U}\star \mathcal{S} \star \mathcal{V}^*,
\end{equation}
where $\mathcal{U}\in \mathbb{C}^{n_1\times n_1\times m}$, $\mathcal{V}\in \mathbb{C}^{n_2\times n_2\times m}$ are unitary tensors, and $\mathcal{S}\in \mathbb{C}^{n_1\times n_2\times m}$ is a diagonal tensor.
\end{theorem}

Note that a diagonal tensor refers that each of its frontal slices is diagonal. 
The tensor tubal rank of a tensor $\mathcal{T}\in \mathbb{C}^{n_1\times n_2\times m}$, denoted as $\text{rank}(\mathcal{T})$, defined as the number of nonzero singular tubes of $\mathcal{S}$ that comes from T-SVD of $\mathcal{T}=\mathcal{U}\star \mathcal{S} \star \mathcal{V}^*$.

\subsection{Main Results}

We first establish the structure of the application of discrete 
quaternion Fourier transform matrix to each tube 
of a real third-order tensor.

\begin{lemma}\label{lem: tsvd_blockReal}
Given $\mathcal{T}\in\mathbb{R}^{n_1\times n_2\times m}$, and a discrete quaternion Fourier transform matrix $\breve{\mathbf{F}}_{\breve{\mu}}$,
\begin{equation}
(\breve{\mathbf{F}}_{\breve{\mu}} \otimes \mathbf{I}_{n_1}) 
\bcirc (\mathcal{T}) (\breve{\mathbf{F}}^*_{\breve{\mu}}\otimes \mathbf{I}_{n_2})= \Bdiag(\breve{\mathcal{D}}),
\end{equation}
where tensor $\breve{\mathcal{D}}$ is computed by applying 
				$\breve{\mathbf{F}}_{\breve{\mu}}$
				along each tube of $\mathcal{T}$, i.e., $\mathcal{T}(s,p,:)$ for $s=1,\cdots n_1$ and $p=1,\cdots n_2$.
\end{lemma}

\begin{proof}
From $\bcirc(\mathcal{T})$, we have
$$
\bcirc(\mathcal{T})=(\mathbf{I}_m\otimes \mathbf{T}_1)+(\mathbf{J}\otimes \mathbf{T}_2)+\cdots+\big((\mathbf{J})^{m-1}\otimes \mathbf{T}_m\big),
$$
where $\mathbf{J}$ is defined in \eqref{J}. 
Let $\breve{\mathbf{\Lambda}}$ be the eigenvalue matrix of $\mathbf{J}$.  According to the derivation of Lemma \ref{lem:6}, we know that $\breve{\mathbf{\Lambda}}=\Diag\big(1,\breve{\omega},\cdots,(\breve{\omega})^{m-1}\big)$  with $\breve{\omega}=\exp \left ( {\displaystyle 
\frac{-2 \pi \breve{\mu}}{m}} \right )$, and  $\breve{\mathbf{F}}^*_{\breve{\mu}}$ is the eigenvector matrix of $\mathbf{J}$.
Note that 
$$
(\breve{\mathbf{F}}_{\breve{\mu}}\otimes \mathbf{I}_{n_1}) (\mathbf{I}_m\otimes \mathbf{T}_1) (\breve{\mathbf{F}}^*_{\breve{\mu}}\otimes \mathbf{I}_{n_2})=\mathbf{I}_m\otimes \mathbf{T}_1,
$$
$$
(\breve{\mathbf{F}}_{\breve{\mu}} \otimes \mathbf{I}_{n_1}) 
\big((\mathbf{J})^t\otimes \mathbf{T}_{t+1}\big) (\breve{\mathbf{F}}^*_{\breve{\mu}}\otimes \mathbf{I}_{n_2})=(\breve{\mathbf{\Lambda}})^t\otimes \mathbf{T}_{t+1},\quad \text{for}\quad t=1,2,\cdots,m-1.$$
Therefore, we obtain 
\begin{eqnarray}
 (\breve{\mathbf{F}}_{\breve{\mu}}\otimes \mathbf{I}_{n_1})\cdot\bcirc(\mathcal{T})\cdot (\breve{\mathbf{F}}^*_{\breve{\mu}}\otimes \mathbf{I}_{n_2}) 
=\mathbf{I}_m\otimes \mathbf{T}_1+\sum^{m-1}_{t=1}\big((\breve{\mathbf{\Lambda}})^t\otimes \mathbf{T}_{t+1}\big)=\Diag\left(
            \breve{\mathbf{D}}_1, \breve{\mathbf{D}}_2, 
						\ldots, \breve{\mathbf{D}}_m
         \right),
\end{eqnarray}
where $\breve{\mathbf{D}}_{t_1}=\mathbf{T}_1+\sum\limits^{m-1}_{t=1}(\breve{\omega}^{t_1-1})^{t}\mathbf{T}_{t+1}$,  $t_1=1,2,\cdots,m$.

On the other hand, we apply 
$\breve{\mathbf{F}}_{\breve{\mu}}$ along the third mode of 
${\cal T}$ and let the resulting tensor be $\breve{\mathcal{D}}$. Then we have
$$
\breve{\mathcal{D}}=\fold(\left(
 \begin{array}{c}
           \mathbf{T}_1+\mathbf{T}_2+\cdots+\mathbf{T}_m  \\
           \mathbf{T}_1+\breve{\omega}\mathbf{T}_2+\cdots+\breve{\omega}^{m-1}\mathbf{T}_m \\
           \vdots\\
            \mathbf{T}_1+(\breve{\omega})^{m-1}\mathbf{T}_2+\cdots+(\breve{\omega}^{m-1})^{m-1}\mathbf{T}_m 
           \end{array}
         \right)).
$$
It is easy to verify that $\breve{\mathbf{D}}_{t_1}$  $(t_1=1,2,\cdots,m)$ are the $t_1$-th frontal slice of tensor $\breve{\mathcal{D}}$. 
The results follow.
\end{proof}

Similar to Lemma 6, the main diagonal structure of 
applying discrete quaternion Fourier matrix to real tensor is preserved.

On the other hand, similar to Lemma \ref{lemm:Minmu}, any quaternion tensor $\breve{\mathcal{T}}$ in three-axis system 
$(\tt{i},\tt{j},\tt{k})$ can be represented in three-axis system $(\breve{\mu},\breve{\alpha},\breve{\beta})$.
Hence in the following, we will focus on quaternion tensors in three-axis system $(\breve{\mu},\breve{\alpha},\breve{\beta})$. 

\begin{theorem}\label{thm:QFT-block}
Given discrete quaternion Fourier transform matrix 
$\breve{\mathbf{F}}_{\breve{\mu}}$, for any quaternion tensor $\breve{\mathcal{T}}\in \mathbb{Q}^{n_1 \times n_2\times m}$ in three-axis system $(\breve{\mu},\breve{\alpha},\breve{\beta})$, i.e.,
$$
\breve{\mathcal{T}}=\mathcal{T}^{(0)}+\mathcal{T}^{(1)} \breve{\mu}+\mathcal{T}^{(2)} \breve{\alpha}+\mathcal{T}^{(3)} \breve{\beta},
$$
where $\mathcal{T}^{(l)}\in \mathbb{R}^{n_1\times n_2\times m}$, $l=0,1,2,3$. Then 
\begin{equation}
(\breve{\mathbf{F}}_{\breve{\mu}}\otimes \mathbf{I}_{n_1} )
\bcirc(\breve{\mathcal{T}}) (\breve{\mathbf{F}}^*_{\breve{\mu}}\otimes \mathbf{I}_{n_2})=  \Bdiag  (\breve{\mathcal{D}}^{(0)})+ \Bdiag (\breve{\mathcal{D}}^{(1)}) \breve{\mu}+\Bdiag  (\breve{\mathcal{D}}^{(2)}) \mathbf{Z}\breve{\alpha}+\Bdiag(\breve{\mathcal{D}}^{(3)}) \mathbf{Z}\breve{\beta},
\end{equation}
where 
$$\mathbf{Z}= \mathbf{A}\otimes \mathbf{I}_{n_2}, \quad \Bdiag  (\breve{\mathcal{D}}^{(l)})=\Diag(\breve{\mathbf{D}}^{(l)}_1,\ldots,\breve{\mathbf{D}}^{(l)}_m),\quad l=0,1,2,3.
$$
$\mathbf{A}$ is defined in \eqref{P}, and $\breve{\mathbf{D}}^{(l)}_t$ ($t=1,2,\cdots,m$) are the frontal slices of tensor $\breve{\mathcal{D}}^{(l)}$ computed by applying 
$\breve{\mathbf{F}}_{\breve{\mu}}$
along each tube of $\breve{\mathcal{T}}^{(l)}$ for $l=0,1,2,3$.
\end{theorem}

\begin{proof}
Let $\bcirc(\breve{\mathcal{T}})=\mathbf{X}^{(0)}+ \mathbf{X}^{(1)} \breve{\mu} +\mathbf{X}^{(2)} \breve{\alpha}+\mathbf{X}^{(3)}\breve{\beta}$, where $\mathbf{X}^{(0)}, \mathbf{X}^{(1)}, \mathbf{X}^{(2)}, \mathbf{X}^{(3)}\in \mathbb{R}^{n_1m\times n_2m}$.
We have, 
\begin{eqnarray*}
&&(\breve{\mathbf{F}}_{\breve{\mu}}\otimes \mathbf{I}_{n_1})
\bcirc(\breve{\mathcal{T}}) (\breve{\mathbf{F}}^*_{\breve{\mu}}\otimes \mathbf{I}_{n_2})\\
&=&(\breve{\mathbf{F}}_{\breve{\mu}}\otimes \mathbf{I}_{n_1})
(\mathbf{X}^{(0)} + \mathbf{X}^{(1)} \breve{\mu}+\mathbf{X}^{(2)} \breve{\alpha} +\mathbf{X}^{(3)} \breve{\beta}) (\breve{\mathbf{F}}^*_{\breve{\mu}}\otimes \mathbf{I}_{n_2})\\
&=& (\breve{\mathbf{F}}_{\breve{\mu}}\otimes \mathbf{I}_{n_1}) 
\mathbf{X}^{(0)} 
(\breve{\mathbf{F}}^*_{\breve{\mu}}\otimes \mathbf{I}_{n_2})+(\breve{\mathbf{F}}_{\breve{\mu}}\otimes \mathbf{I}_{n_1})
 \mathbf{X}^{(1)} 
 (\breve{\mathbf{F}}^*_{\breve{\mu}}\otimes \mathbf{I}_{n_2})\breve{\mu}\\
&+& (\breve{\mathbf{F}}_{\breve{\mu}}\otimes \mathbf{I}_{n_1}) 
\mathbf{X}^{(2)} 
(\breve{\mathbf{F}}^*_{\breve{\mu}}\otimes \mathbf{I}_{n_2})(\breve{\mathbf{F}}_{\mu}^2\otimes \mathbf{I}_{n_2})\breve{\alpha}+(\breve{\mathbf{F}}_{\mu}\otimes \mathbf{I}_{n_1}) \mathbf{X}^{(3)} (\breve{\mathbf{F}}^*_{\mu}\otimes \mathbf{I}_{n_2})(\breve{\mathbf{F}}_{\breve{\mu}}^2\otimes \mathbf{I}_{n_2})\breve{\beta}\\
&=& \Bdiag  (\breve{\mathcal{D}}^{(0)})+ \Bdiag (\breve{\mathcal{D}}^{(1)}) \breve{\mu}+\Bdiag(\breve{\mathcal{D}}^{(2)}) \mathbf{Z}\breve{\alpha}+\Bdiag(\breve{\mathcal{D}}^{(3)}) \mathbf{Z}\breve{\beta}.
\end{eqnarray*}
The second equation is established due to Lemma \ref{lem:qua-commu}, and the last equation holds because of Lemma \ref{lem: tsvd_blockReal}. 
Here $\mathbf{Z}= \mathbf{A}\otimes \mathbf{I}_{n_2}$, where 
$\mathbf{A}$ is defined in \eqref{P}. Also $ \Bdiag (\breve{\mathcal{D}}^{(l)})=\Diag(\breve{\mathbf{D}}^{(l)}_1, \breve{\mathbf{D}}^{(l)}_2,
\ldots,\breve{\mathbf{D}}^{(l)}_m),$ and $\breve{\mathbf{D}}^{(l)}_t$ ($t=1,2,\cdots,m$) are the frontal slices of tensor $\breve{\mathcal{D}}^{(l)}$ computed by applying 
$\breve{\mathbf{F}}_{\breve{\mu}}$
along each tube of $\breve{\mathcal{T}}^{(l)}$ for $l=0,1,2,3$.
The results follow.
\end{proof}


To start the discussion on quaternion tensor singular value decomposition, we first introduce quaternion singular value decomposition in the following lemma. 
\begin{lemma}\textbf{[Quaternion Singular Value Decomposition(QSVD)]}\cite{zhang1997quaternions,jia2019lanczos}
Given any quaternion matrix $\breve{\mathbf{M}}\in \mathbb{Q}^{n_1\times n_2}$, then there exist two unitary quaternion matrices
$\breve{\mathbf{U}}\in \mathbb{Q}^{n_1\times n_1}$ and $\breve{\mathbf{V}}\in \mathbb{Q}^{n_2\times n_2}$, i.e.,
$\breve{\mathbf{U}}^*\breve{\mathbf{U}} =\breve{\mathbf{U}} \breve{\mathbf{U}}^*=\mathbf{I}_{n_1}$, 
$\breve{\mathbf{V}}^*\breve{\mathbf{V}} =\breve{\mathbf{V}} \breve{\mathbf{V}}^*=\mathbf{I}_{n_2}$,
such that 
\begin{equation}
\breve{\mathbf{U}}^*\breve{\mathbf{M}}\breve{\mathbf{V}}=\mathbf{\Sigma},
\end{equation}
where $\mathbf{\Sigma} \in \mathbb{R}^{n_1\times n_2}$, with 
$\mathbf{\Sigma}_{sp}=0$ when $s\neq p$; and $\mathbf{\Sigma}_{ss} \ge 0$ 
for $s=1,2,\cdots,\min(m,n)$.
\end{lemma}
We now formally present quaternion tensor singular value decomposition.  
\begin{theorem}\textbf{[Quaternion Tensor SVD (QT-SVD)]}\label{Thm3}
Given any quaternion tensor $\breve{\mathcal{T}}\in \mathbb{Q}^{n_1 \times n_2\times m}$,  then there exist unitary quaternion tensors 
$\breve{\mathcal{U}}\in \mathbb{Q}^{n_1\times n_1\times m}$ and $\breve{\mathcal{V}}\in \mathbb{Q}^{n_2\times n_2\times m}$, and a diagonal tensor $\breve{\mathcal{S}}\in \mathbb{Q}^{n_1\times n_2\times m}$ such that 
\begin{equation}\label{eq:qtsvd}
\breve{\mathcal{T}}=\breve{\mathcal{U}}\star \breve{\mathcal{S}}\star \breve{\mathcal{V}}^*.
\end{equation}  
The decomposition \eqref{eq:qtsvd} is called the quaternion tensor singular value decomposition of 
$\breve{\mathcal{T}}$.
\end{theorem}

\begin{proof}
We transform $\bcirc(\breve{\mathcal{T}})$ into the Fourier domain
by using $\breve{\mathbf{F}}_{\breve{\mu}}$, 
\begin{equation}
(\breve{\mathbf{F}}_{\breve{\mu}}\otimes \mathbf{I}_{n_1})
\bcirc(\breve{\mathcal{T}})
(\breve{\mathbf{F}}^*_{\breve{\mu}}\otimes \mathbf{I}_{n_2})=\breve{\mathbf{D}}.\end{equation}
From Theorem \ref{thm:QFT-block}, we have 
$$
\breve{\mathbf{D}}=\Bdiag (\breve{\mathcal{D}}^{(0)})+ \Bdiag (\breve{\mathcal{D}}^{(1)}) \breve{\mu}+\Bdiag(\breve{\mathcal{D}}^{(2)}) \mathbf{Z}\breve{\alpha}+\Bdiag (\breve{\mathcal{D}}^{(3)}) \mathbf{Z}\breve{\beta}, 
$$
where
$\Bdiag(\breve{\mathcal{D}}^{(l)})=\Diag(\breve{\mathbf{D}}^{(l)}_1,\cdots,\breve{\mathbf{D}}^{(l)}_m),~ l=0,1,2,3$. More precisely, the structure of $\breve{\mathbf{D}}$
is given as follows.
\begin{equation} \label{st1}
\text{If  m is even:}\quad   \breve{\mathbf{D}} =
\left(
           \begin{array}{cccccc}
\breve{\mathbf{D}}_{1} & 0 & \cdots  &  \cdots  & \cdots & 0 \\
0 & \breve{\mathbf{D}}_{1,2} & 0         & \cdots & 0 & \breve{\mathbf{D}}_{2,2} \\
   & 0 & \ddots &           & \begin{sideways} $\ddots$ \end{sideways} & 0 \\
\vdots   & \vdots &    &    \breve{\mathbf{D}}_{m/2+1}     &    &  \vdots      \\
   & 0 & \begin{sideways} $\ddots$ \end{sideways} & & \ddots & 0 \\ 
0  & \breve{\mathbf{D}}_{1,m} & 0        & \cdots & 0 & \breve{\mathbf{D}}_{2,m} \\
   \end{array}
	\right ),
	\end{equation}
 where 	
\begin{eqnarray*}
 \text{ $t=1, ~\frac{m}{2}+1$:}\quad 
&\breve{\mathbf{D}}_{t}=\breve{\mathbf{D}}^{(0)}_t+\breve{\mathbf{D}}^{(1)}_t\breve{\mu}+\breve{\mathbf{D}}^{(2)}_t\breve{\alpha}+\breve{\mathbf{D}}^{(3)}_t\breve{\beta}, \\
 \text{  $t=2,\ldots,\frac{m}{2},\frac{m}{2}+2,\ldots,m$:}\quad &  \breve{\mathbf{D}}_{1,t}=\breve{\mathbf{D}}^{(0)}_t+\breve{\mathbf{D}}^{(1)}_t\breve{\mu}, \quad
\breve{\mathbf{D}}_{2,t}=
\breve{\mathbf{D}}^{(2)}_t\breve{\alpha}+\breve{\mathbf{D}}^{(3)}_t\breve{\beta}. 
\end{eqnarray*}
Similarly,
\begin{equation} \label{st2}
\text{if $n$ is odd:}\quad  \breve{\mathbf{D}} =
\left(
           \begin{array}{ccccccc}
\breve{\mathbf{D}}_{1} & 0 & \cdots  &  \cdots & \cdots  & \cdots & 0 \\
0 & \breve{\mathbf{D}}_{1,2} & 0         & \cdots & \cdots & 0 & \breve{\mathbf{D}}_{2,2} \\
   & 0 & \ddots &       &    & \begin{sideways} $\ddots$ \end{sideways} & 0 \\
\vdots   & \vdots &    &    \breve{\mathbf{D}}_{1,(m+1)/2}  & 
\breve{\mathbf{D}}_{2,(m+1)/2}
   &    &  \vdots      \\
	\vdots   & \vdots &    &    \breve{\mathbf{D}}_{1,(m+1)/2+1}  & 
\breve{\mathbf{D}}_{2,(m+1)/2+1}
   &    &  \vdots      \\
   & 0 & \begin{sideways} $\ddots$ \end{sideways} & & & \ddots & 0 \\ 
0  & \breve{\mathbf{D}}_{1,m} & 0    & \cdots    & \cdots & 0 & \breve{\mathbf{D}}_{2,m} \\
   \end{array}
	\right ),
	\end{equation}
where 	
\begin{eqnarray*}
\text{ $t=1$:}\quad &\breve{\mathbf{D}}_{t}=\breve{\mathbf{D}}^{(0)}_t+\breve{\mathbf{D}}^{(1)}_t\breve{\mu}+\breve{\mathbf{D}}^{(2)}_t\breve{\alpha}+\breve{\mathbf{D}}^{(3)}_t\breve{\beta},  
 \\
\text{  $t=2,\ldots,m$:}\quad  &\breve{\mathbf{D}}_{1,t}=\breve{\mathbf{D}}^{(0)}_t+\breve{\mathbf{D}}^{(1)}_t\breve{\mu}, \quad
\breve{\mathbf{D}}_{2,t}=
\breve{\mathbf{D}}^{(2)}_t\breve{\alpha}+\breve{\mathbf{D}}^{(3)}_t\breve{\beta}.
\end{eqnarray*}
We remark that the size of $\breve{\mathbf{D}}_{1,t}$
and $\breve{\mathbf{D}}_{2,t}$ in (\ref{st1}) and (\ref{st2}) 
is $n_1 \times n_2$.

According to Theorem 1, we can perform block-diagonalization for 
$\breve{\mathbf{D}}$ by permuting of a set of block rows and block columns. 
More precisely, 
\begin{itemize}
\item  If $m$ is  even: we permute $\breve{\mathbf{D}}$ by exchanging its $n_1$ rows at $(k-1)n_1+1:k n_1$
with the rows at $(m-k+2)n_1+1:(m-k+3)n_1$,
and the $n_2$ columns at
$(k-1)n_2+1:k n_2$ 
with the columns at $(m-k+2)n_2+1:(m-k+3)n_2$  
for $k=3,4,\ldots,\frac{m}{2}$.

\item If $m$ is odd: we permute $\breve{\mathbf{D}}$ by exchanging its $n_1$ rows at $(k-1)n_1+1:k n_1$ 
with the rows at $(m-k+2)n_1+1:(m-k+3)n_1$,
and exchanging the $n_2$ columns at
$(k-1)n_2+1:k n_2$ 
with the columns at $(m-k+2)n_2+1:(m-k+3)n_2$  
for $k=3,4,\ldots,\frac{m+1}{2}$.
\end{itemize}

Then we can construct 1-by 1 and 2-by-2 block matrices in the following way.

\noindent If $m$ is even, 
\begin{eqnarray}
\text{ t=1:} \quad  &\breve{\mathbf{G}}_{t} = \breve{\mathbf{D}}_{1},\label{eq:G1even0}\\
\text{  $t=2,\ldots,\frac{m}{2}$:}\quad  
&\breve{\mathbf{G}}_t=\left(
           \begin{array}{cc}
            \breve{\mathbf{D}}_{1,t} &  \breve{\mathbf{D}}_{2,t}   \\
  \breve{\mathbf{D}}_{1,m+2-t}   & \breve{\mathbf{D}}_{2,m+2-t} \\
           \end{array}
         \right),
				\label{eq:Gteven}\\
	\text{ $t=\frac{m}{2}+1$:}\quad  &\breve{\mathbf{G}}_{t}  =\breve{\mathbf{D}}_{m/2+1}.\label{eq:G1even} 
\end{eqnarray}
If $m$ is odd, 
\begin{eqnarray}
\text{  $t=1$:}\quad  &\breve{\mathbf{G}}_{t} =\breve{\mathbf{D}}_{1},\label{eq:G1}\\
\text{ $t=2,\ldots,\frac{m+1}{2}$:}\quad  &
\breve{\mathbf{G}}_t=\left(
           \begin{array}{cc}
            \breve{\mathbf{D}}_{1,t} &  \breve{\mathbf{D}}_{2,t}   \\
  \breve{\mathbf{D}}_{1,m+2-t}   & \breve{\mathbf{D}}_{2,m+2-t} \\
           \end{array}
         \right). 	\label{eq:Gt}			
\end{eqnarray}
We remark that the even case and the odd case are similar.
For simplicity, we only consider 
the odd case in the following discussion. 

We apply the quaternion SVD on all the blocks $\breve{\mathbf{G}}_t$ for $t=1,\ldots, \frac{m+1}{2}$. Precisely,
\begin{eqnarray*}
t=1:\quad &\breve{\mathbf{G}}_t=\breve{\mathbf{U}}_1 \mathbf{\Sigma}_1
\breve{\mathbf{V}}^*_1,&
\\
t=2,\cdots, \frac{m+1}{2}:\quad &\breve{\mathbf{G}}_t=\breve{\mathbf{U}}_t \mathbf{\Lambda}_t\breve{\mathbf{V}}^*_t, &\quad 
\text{where}
\end{eqnarray*}
\begin{equation}\label{UtVt}
\breve{\mathbf{U}}_t=\left(
           \begin{array}{cc}
            \breve{\mathbf{U}}_{1,t} &  \breve{\mathbf{U}}_{2,t}   \\
  \breve{\mathbf{U}}_{1,m+2-t}   & \breve{\mathbf{U}}_{2,m+2-t} \\
           \end{array}
         \right), \quad \mathbf{\Lambda}_t=\left(
           \begin{array}{cc}
             \mathbf{\Sigma}_{t} &   0 \\
  0  &  \mathbf{\Sigma}_{m+2-t} \\
           \end{array}
         \right), \quad \breve{\mathbf{V}}_t=\left(
           \begin{array}{cc}
            \breve{\mathbf{V}}_{1,t} &  \breve{\mathbf{V}}_{2,t}   \\
  \breve{\mathbf{V}}_{1,m+2-t}   & \breve{\mathbf{V}}_{2,m+2-t} \\
           \end{array}
         \right).    
\end{equation}
It implies that $ \breve{\mathbf{D}}$ admits the following decomposition:
$$         
         \breve{\mathbf{D}}=\breve{\mathbf{U}}
				\mathbf{\Sigma}\breve{\mathbf{V}}^*,
$$         
where   
$$
\mathbf{\Sigma}=\Diag(\mathbf{\Sigma}_1,\mathbf{\Sigma}_2,\cdots,\mathbf{\Sigma}_m),
$$
\begin{equation}\label{eq:USV}
\breve{\mathbf{U}}=\left(
           \begin{array}{ccccc}
            \breve{\mathbf{U}}_1 & & &  & \\
          & \breve{\mathbf{U}}_{1,2}& &  &\breve{\mathbf{U}}_{2,2} \\
           & &\ddots& \begin{sideways} $\ddots$ \end{sideways}& \\
  & &\begin{sideways} $\ddots$ \end{sideways}&   \ddots  & \\
     & \breve{\mathbf{U}}_{1,m}& & & \breve{\mathbf{U}}_{2,m}\\
           \end{array}
         \right); \quad   
\breve{\mathbf{V}}= \left(
           \begin{array}{ccccc}
            \breve{\mathbf{V}}_1 & & &  & \\
          & \breve{\mathbf{V}}_{1,2}& &  &\breve{\mathbf{V}}_{2,2} \\
           & &\ddots& \begin{sideways} $\ddots$ \end{sideways}& \\
  & &\begin{sideways} $\ddots$ \end{sideways}&   \ddots  & \\
     & \breve{\mathbf{V}}_{1,m}& & & \breve{\mathbf{V}}_{2,m}\\
           \end{array}
         \right).
\end{equation}
Then,
\begin{equation}\label{DUSV}
(\breve{\mathbf{F}}^*_{\breve{\mu}}\otimes \mathbf{I}_{n_1})
\breve{\mathbf{D}} (\breve{\mathbf{F}}_{\breve{\mu}}\otimes \mathbf{I}_{n_2})=(\breve{\mathbf{F}}^*_{\breve{\mu}}\otimes \mathbf{I}_{n_1})\breve{\mathbf{U}}(\breve{\mathbf{F}}_{\breve{\mu}}\otimes \mathbf{I}_{n_1})(\breve{\mathbf{F}}^*_{\breve{\mu}}\otimes \mathbf{I}_{n_1})\mathbf{\Sigma}(\breve{\mathbf{F}}_{\breve{\mu}}\otimes \mathbf{I}_{n_2})(\breve{\mathbf{F}}^*_{\breve{\mu}}\otimes \mathbf{I}_{n_2})\breve{\mathbf{V}}(\breve{\mathbf{F}}_{\breve{\mu}}\otimes \mathbf{I}_{n_2}).
\end{equation}
The equation holds since $(\breve{\mathbf{F}}_{\breve{\mu}}\otimes \mathbf{I})(\breve{\mathbf{F}}^*_{\breve{\mu}}\otimes \mathbf{I})=\mathbf{I}$, where $\mathbf{I}$ is the identity matrix of appropriate size.

From Lemma \ref{lem: tsvd_blockReal} and Theorem \ref{thm:QFT-block}, we deduce that equation \eqref{DUSV} results in the product of three block circulant matrices, i.e.,
\begin{equation}
\bcirc(\breve{\mathcal{T}})=\bcirc(\breve{\mathcal{U}})
\bcirc(\breve{\mathcal{S}}) \bcirc(\breve{\mathcal{V}}^*),
\end{equation}
which implies 
$$
\breve{\mathcal{T}}=\breve{\mathcal{U}}\star 
\breve{\mathcal{S}}\star \breve{\mathcal{V}}^*.
$$

Now the remaining problem is to prove $\breve{\mathcal{U}}$ and $\breve{\mathcal{V}}$ are unitary. 
It is equivalent to the equation
$\bcirc(\breve{\mathcal{U}}^*)  \bcirc(\breve{\mathcal{U}})=\mathbf{I}$. From Theorem \ref{thm:QFT-block}, we have
$$
\bcirc(\breve{\mathcal{U}}^*)  \bcirc(\breve{\mathcal{U}})=(\breve{\mathbf{F}}^*_{\breve{\mu}}\otimes \mathbf{I}_{n_1})\breve{\mathbf{U}}^*(\breve{\mathbf{F}}_{\breve{\mu}}\otimes \mathbf{I}_{n_1})(\breve{\mathbf{F}}^*_{\breve{\mu}}\otimes \mathbf{I}_{n_1})\breve{\mathbf{U}}(\breve{\mathbf{F}}_{\breve{\mu}}\otimes \mathbf{I}_{n_1})=(\breve{\mathbf{F}}^*_{\breve{\mu}}\otimes \mathbf{I}_{n_1})\breve{\mathbf{U}}^*\breve{\mathbf{U}}(\breve{\mathbf{F}}_{\breve{\mu}}\otimes \mathbf{I}_{n_1}).
$$
Note that $\breve{\mathbf{U}}^*\breve{\mathbf{U}}=\mathbf{I}$, we deduce that $\bcirc(\breve{\mathcal{U}}^*)  \bcirc(\breve{\mathcal{U}})=\mathbf{I}$, which implies $\breve{\mathcal{U}}^*\star\breve{\mathcal{U}}=\mathcal{I}$. Similarly, $\breve{\mathcal{U}}\star\breve{\mathcal{U}}^*=\mathcal{I}$, hence $\breve{\mathcal{U}}$ is unitary. By using  similar arguments, we can 
show $\breve{\mathcal{V}}$ is unitary. 
\end{proof}
Given
\begin{equation}\label{eq:LWR}
\breve{\mathcal{L}}=\fold\Big(\left(
           \begin{array}{c}
            \breve{\mathbf{U}}_1 \\ \breve{\mathbf{U}}_{1,2}+ \breve{\mathbf{U}}_{2,2} \\
            \vdots \\ \breve{\mathbf{U}}_{1,m}+ \breve{\mathbf{U}}_{2,m} \\
           \end{array}
         \right)\Big);\quad 
          \mathcal{W} =\fold\Big(\left(
           \begin{array}{c}
             \mathbf{\Sigma}_1 \\  \mathbf{\Sigma}_{2}\\
            \vdots \\ \mathbf{\Sigma}_{m} \\
           \end{array}
         \right)\Big);\quad
         \breve{\mathcal{R}}=\fold\Big(\left(
           \begin{array}{c}
            \breve{\mathbf{V}}_1 \\ \breve{\mathbf{V}}_{1,2}+ \breve{\mathbf{V}}_{2,2} \\
            \vdots \\ \breve{\mathbf{V}}_{1,m}+ \breve{\mathbf{V}}_{2,m} \\
           \end{array}
         \right)\Big);
\end{equation}
we remark that the frontal slices of tensors $\breve{\mathcal{L}}$, $\mathcal{W}$  and $\breve{\mathcal{R}}$ are computed by applying QFFT along each tube of $\breve{\mathcal{U}}$, $\breve{\mathcal{S}}$, 
$\breve{\mathcal{V}}$ respectively, which means that $\breve{\mathcal{U}}$, $\breve{\mathcal{S}}$, 
$\breve{\mathcal{V}}$ can be given by applying 
inverse QFFT along each tube of $\breve{\mathcal{L}}$, $\mathcal{W}$  and $\breve{\mathcal{R}}$ respectively.

\begin{remark}
Given quaternion tensor $\breve{\mathcal{T}}\in \mathbb{Q}^{n_1\times n_2\times m}$,  for any orthogonal quaternion tensor $\breve{\mathcal{Q}}\in \mathbb{Q}^{n_1\times n_1\times m}$, i.e.,
$\breve{\mathcal{Q}}^* \star \breve{\mathcal{Q}} =\breve{\mathcal{Q}}\star \breve{\mathcal{Q}}^* =\mathbf{I}$, we remark that
\begin{equation}
\|\breve{\mathcal{Q}}\star \breve{\mathcal{T}}\|^2_F=\|\bcirc(\mathcal{\breve{\mathcal{Q}}}) \unfold(\breve{\mathcal{T}})\|^2_F=\|\breve{\mathcal{T}}\|^2_F.
\end{equation}
\end{remark}

As we know that the standard SVD gives the best low rank approximation for any matrix, in the following, a similar result will be presented for 
quaternion tensor singular value decomposition. 
To derive the best rank-$r$ approximation for $\breve{\mathcal{T}}$, 
i.e., find a $\breve{\mathcal{Z}}\in \Omega$ such that
$$
\mathop{\min}\limits_{\breve{\mathcal{Z}}\in \Omega}
\|\breve{\mathcal{T}}-\breve{\mathcal{Z}}\|^2_F,
$$ 
where 
$$
\Omega=\{\breve{\mathcal{Z}}|\breve{\mathcal{Z}}=\breve{\mathcal{X}}
\star \breve{\mathcal{Y}}, ~ \breve{\mathcal{X}}\in \mathbb{Q}^{n_1\times r\times m}, ~ \breve{\mathcal{Y}}\in \mathbb{Q}^{r\times n_2\times m} \}.
$$

Let us revisit the analysis in Theorem \ref{Thm3}. 
In the proof for Theorem  \ref{Thm3}, quaternion SVD is utilized to get the quaternion tensor singular value decomposition: 
$(\breve{\mathcal{U}},\mathcal{S},\breve{\mathcal{V}})$ for quaternion 
tensor $\breve{\mathcal{T}}$. In the following, we consider their truncated 
versions with keeping $r$ components in the corresponding terms: 
\begin{equation}\label{Dr}
\breve{\mathbf{D}}^{[r]}=\breve{\mathbf{U}}^{[r]}
\mathbf{\Sigma}^{[r]}(\breve{\mathbf{V}}^{[r]})^*, 
\end{equation} 
where 
$$
\mathbf{\Sigma}^{[r]} :=\Diag(\mathbf{\Sigma}_1^{[r]},\mathbf{\Sigma}^{[r]}_2,\cdots,\mathbf{\Sigma}^{[r]}_m),
$$
$$ \breve{\mathbf{U}}^{[r]}:=\left(
           \begin{array}{ccccc}
            \breve{\mathbf{U}}^{[r]}_1 & & &  & \\
          & \breve{\mathbf{U}}^{[r]}_{1,2}& &  &\breve{\mathbf{U}}^{[r]}_{2,2} \\
           & &\ddots& \begin{sideways} $\ddots$ \end{sideways}& \\
  & &\begin{sideways} $\ddots$ \end{sideways}&   \ddots  & \\
     & \breve{\mathbf{U}}^{[r]}_{1,m}& & & \breve{\mathbf{U}}^{[r]}_{2,m}\\
           \end{array}
         \right), \quad   
\breve{\mathbf{V}}^{[r]}:= \left(
           \begin{array}{ccccc}
            \breve{\mathbf{V}}^{[r]}_1 & & &  & \\
          & \breve{\mathbf{V}}^{[r]}_{1,2}& &  &\breve{\mathbf{V}}^{[r]}_{2,2} \\
           & &\ddots& \begin{sideways} $\ddots$ \end{sideways}& \\
  & &\begin{sideways} $\ddots$ \end{sideways}&   \ddots  & \\
     & \breve{\mathbf{V}}^{[r]}_{1,m}& & & \breve{\mathbf{V}}^{[r]}_{2,m}\\
           \end{array}
         \right),
$$
and the components of $\breve{\mathbf{U}}^{[r]}$, $\mathbf{\Sigma}^{[r]}$ and $\breve{\mathbf{V}}^{[r]}$ are given as follows.

For the case when $m$ is odd,
\begin{eqnarray*} 
t=1: &\quad 
\breve{\mathbf{U}}^{[r]}_t \doteq\breve{\mathbf{U}}_t(:,1:r),\quad 
 \mathbf{\Sigma}^{[r]}_t  \doteq \mathbf{\Sigma}_t(1:r,1:r),\quad 
 \breve{\mathbf{V}}^{[r]}_t  \doteq\breve{\mathbf{V}}_t(:,1:r);\\
t=2,\cdots,\frac{m+1}{2}:& \quad \left(
           \begin{array}{cc}
            \breve{\mathbf{U}}^{[r]}_{1,t} &  \breve{\mathbf{U}}^{[r]}_{2,t}   \\
  \breve{\mathbf{U}}^{[r]}_{1,m+2-t}   & \breve{\mathbf{U}}^{[r]}_{2,m+2-t} \\
           \end{array}
         \right) \doteq \breve{\mathbf{U}}_{t}(:,1:2r); \\      
       &  \left(
           \begin{array}{cc}
             \mathbf{\Sigma}^{[r]}_{t} & 0   \\
0    &  \mathbf{\Sigma}^{[r]}_{m+2-t} \\
           \end{array}
         \right) \doteq \mathbf{\Lambda}_{t}(:,1:2r);\\    
         &     
         \left(
           \begin{array}{cc}
            \breve{\mathbf{V}}^{[r]}_{1,t} &  \breve{\mathbf{V}}^{[r]}_{2,t}   \\
  \breve{\mathbf{V}}^{[r]}_{1,m+2-t}   & \breve{\mathbf{V}}^{[r]}_{2,m+2-t} \\
           \end{array}
         \right) \doteq \breve{\mathbf{V}}_{t}(:,1:2r).    
\end{eqnarray*}
For the case when $m$ is even, besides the blocks above, there is one more block given below:
$$
\breve{\mathbf{U}}^{[r]}_{\frac{m+2}{2}} =\breve{\mathbf{U}}_{\frac{m+2}{2}}(:,1:r), \quad
\mathbf{\Sigma}^{[r]}_{\frac{m+2}{2}}=
\mathbf{\Sigma}_{\frac{m+2}{2}}(:,1:r);\quad\breve{\mathbf{V}}^{[r]}_{\frac{m+2}{2}}=\breve{\mathbf{V}}_{\frac{m+2}{2}}(:,1:r).
$$
Define $(\breve{\mathcal{U}}^{[r]},\breve{\mathcal{S}}^{[r]},
\breve{\mathcal{V}}^{[r]})$ such that 
\begin{eqnarray}
\bcirc(\breve{\mathcal{S}}^{[r]})&=&(\breve{\mathbf{F}}^*_{\breve{\mu}}\otimes \mathbf{I}_{r}) \mathbf{\Sigma}^{[r]}(\breve{\mathbf{F}}_{\breve{\mu}}\otimes \mathbf{I}_{r});\label{Sr}\\
 \bcirc(\breve{\mathcal{U}}^{[r]})&=&(\breve{\mathbf{F}}^*_{\breve{\mu}}\otimes \mathbf{I}_{n_1})\breve{\mathbf{U}}^{[r]}(\breve{\mathbf{F}}_{\breve{\mu}}\otimes \mathbf{I}_{r});\label{Ur}\\
  \bcirc(\breve{\mathcal{V}}^{[r]})&=&(\breve{\mathbf{F}}^*_{\breve{\mu}}\otimes \mathbf{I}_{n_2})\breve{\mathbf{V}}^{[r]}(\breve{\mathbf{F}}_{\breve{\mu}}\otimes \mathbf{I}_{r}).\label{Vr}
\end{eqnarray} 
Now we can derive the best rank-$r$ approximation for the quaternion tensor 
$\breve{\mathcal{T}}$ in the following theorem.

\begin{theorem}[Best low-rank approximation]
Given a quaternion tensor $\breve{\mathcal{T}}\in \mathbb{Q}^{n_1\times n_2\times m}$, its quaternion T-SVD is given by 
$\breve{\mathcal{T}}=\breve{\mathcal{U}} \star
\breve{\mathcal{S}}\star\breve{\mathcal{V}}^*$. For $r<\min(n_1,n_2)$, define
$$
\breve{\mathcal{T}}^{[r]}= \breve{\mathcal{U}}^{[r]} \star  
\breve{\mathcal{S}}^{[r]} \star  (\breve{\mathcal{V}}^{[r]})^*,
$$
then 
$\breve{\mathcal{T}}^{[r]}=\mathop{\arg\min}\limits_{\breve{\mathcal{Z}}\in \Omega}
\|\breve{\mathcal{T}}-\breve{\mathcal{Z}}\|^2_F,
$ where $\Omega=\{\breve{\mathcal{Z}}|\breve{\mathcal{Z}}=\breve{\mathcal{X}}
\star \breve{\mathcal{Y}}; \breve{\mathcal{X}}\in \mathbb{Q}^{n_1\times r\times m}, \breve{\mathcal{Y}}\in \mathbb{Q}^{r\times n_2\times m} \}$.
\end{theorem}
\begin{proof}
Since $\breve{\mathcal{Z}}=\breve{\mathcal{X}}\star \breve{\mathcal{Y}}$, we have
$$
\bcirc(\breve{\mathcal{Z}})=\bcirc(\breve{\mathcal{X}})
\bcirc( \breve{\mathcal{Y}}).
$$
Then
\begin{eqnarray*}
\|\breve{\mathcal{T}}-\breve{\mathcal{Z}}\|^2_F&=&\frac{1}{m}\|\bcirc(\breve{\mathcal{T}})-\bcirc(\breve{\mathcal{Z}})\|_F\\
&=&\frac{1}{m}\|(\breve{\mathbf{F}}_{\breve{\mu}} 
\otimes \mathbf{I}_{n_1})\big(\bcirc(\breve{\mathcal{T}})-\bcirc(\breve{\mathcal{Z}})\big)(\breve{\mathbf{F}}^*_{\breve{\mu}}\otimes \mathbf{I}_{n_2})\|^2_F\\
&=& \frac{1}{m}\|\breve{\mathbf{D}}-(\breve{\mathbf{F}}_{\breve{\mu}}\otimes \mathbf{I}_{n_1})\bcirc(\breve{\mathcal{Z}})(\breve{\mathbf{F}}
^*_{\breve{\mu}}\otimes \mathbf{I}_{n_2})\|^2_F\\
&=&\frac{1}{m}\|\breve{\mathbf{D}}-\breve{\mathbf{D}}_{\Omega} \|^2_F,
\end{eqnarray*}
where $\breve{\mathbf{D}}_{\Omega}=(\breve{\mathbf{F}}_{\breve{\mu}}
\otimes \mathbf{I}_{n_1})\bcirc(\breve{\mathcal{Z}})(\breve{\mathbf{F}}^*_{\breve{\mu}}\otimes \mathbf{I}_{n_2})$. 

Therefore, instead of finding $\breve{\mathcal{Z}}$ such that $\min\limits_{\breve{\mathcal{Z}}\in \Omega}\|\breve{\mathcal{T}}-\breve{\mathcal{Z}}\|^2_F$, we turn to look for $\breve{\mathbf{D}}_{\Omega}$ such that
$$\min_{\breve{\mathbf{D}}_{\Omega}\in \Omega_{ \mathbf{D}}}\|\breve{\mathbf{D}}-\breve{\mathbf{D}}_{\Omega} \|^2_F,$$
where $\Omega_{ \mathbf{D}}=\{\breve{\mathbf{D}}_{\Omega}|\breve{\mathbf{D}}_{\Omega}=(\breve{\mathbf{F}}_{\breve{\mu}}\otimes \mathbf{I}_{n_1})\bcirc(\breve{\mathcal{Z}})(\breve{\mathbf{F}}^*_{\breve{\mu}}\otimes \mathbf{I}_{n_2}), \breve{\mathcal{Z}}\in \Omega \}$.

Construct $\breve{\mathbf{D}}^{[r]}$ by equation \eqref{Dr}. As it is well-known that quaternion SVD gives the best rank-$r$ approximation, we have
$$\breve{\mathbf{D}}^{[r]}=\mathop{\arg\min}_{\breve{\mathbf{D}}_{\Omega}\in \Omega_{ \mathbf{D}}}\|\breve{\mathbf{D}}-\breve{\mathbf{D}}_{\Omega}\|.$$
It implies that $(\breve{\mathcal{U}}^{[r]},\breve{\mathcal{S}}^{[r]},\breve{\mathcal{V}}^{[r]})$ are given in \eqref{Sr}-\eqref{Vr}. The results hence follow.
\end{proof}

Based on Theorem \ref{thm:QFT-block}, quaternion T-SVD can be implemented by using the quaternion fast Fourier transform, which is presented in Algorithm \ref{algo:qtsvd} below. To simplify the notations, here we directly use the command symbols ''qsvd", ''qfft" and ''iqfft" in MATLAB to represent quaternion SVD, quaternion fast Fourier transform and inverse quaternion fast Fourier transform. 

\begin{algorithm}[h] 
\caption{Fast Quaternion TSVD   \label{algo:qtsvd}}
\begin{algorithmic}[1]
\REQUIRE Given quaternion tensor $\breve{\mathcal{T}}= \mathcal{T}^{(0)}+\mathcal{T}^{(1)} \breve{\mu}+\mathcal{T}^{(2)} \breve{\alpha}+\mathcal{T}^{(3)} \breve{\beta} \in \mathbb{Q}^{n_1\times n_2\times m}$;  quaternion Fourier transform matrix $\breve{\mathbf{F}}_{\mu}$.\\

\ENSURE $(\breve{\mathcal{U}}, \breve{\mathcal{S}},\breve{\mathcal{V}})$ such that $\breve{\mathcal{T}}=\breve{\mathcal{U}}\star\breve{\mathcal{S}}\star\breve{\mathcal{V}}^*$.\\
 
\STATE  Block-diagonalization.\\
$\breve{\mathcal{D}}^{(l)}=\text{qfft}(\mathcal{T}^{(l)},[~],3)$, $l=0,1,2,3$. \\

\STATE Form $\{\breve{\mathbf{G}}_t, t=1,2,\cdots,\lceil \frac{m+1}{2}\rceil \}$  based on \eqref{eq:G1even0}-\eqref{eq:Gt}.

\FOR {$t$ = 1 : $\lceil \frac{m+1}{2}\rceil$ }
\STATE $[\breve{\mathbf{U}}_t, \mathbf{\Lambda}_t,\breve{\mathbf{V}}_t]=\text{qsvd}(\breve{\mathbf{G}}_t)$.\\
\ENDFOR 

\STATE Construct $(\breve{\mathbf{U}},\mathbf{\Sigma},\breve{\mathbf{V}})$ based on $\eqref{eq:USV}$.\\

\STATE Construct tensors 
 $(\breve{\mathcal{L}},\mathcal{W}, \breve{\mathcal{R}}) $
from \eqref{eq:LWR}.

\STATE  $\breve{\mathcal{U}}=\text{iqfft}(\breve{\mathcal{L}},[~],3)$;
$\quad \breve{\mathcal{S}} =\text{iqfft}( \mathcal{W},[~],3)$; $\quad\breve{\mathcal{V}}=\text{iqfft}(\breve{\mathcal{R}},[~],3)$.
\end{algorithmic}
\end{algorithm}

\begin{remark}
In line 2-3 of Algorithm  \ref{algo:qtsvd}, $\lceil \frac{m+1}{2}\rceil $ denotes the round operator that rounds $\frac{m+1}{2}$ up to the nearest integer. Note that Algorithm \ref{algo:qtsvd} performs a "qfft" operation on tensor $\breve{\mathcal{T}}$, and one "iqfft" operation on each of tensors $\breve{\mathcal{L}}$, $\mathcal{W}$ and $\breve{\mathcal{R}}$, and $\lceil \frac{m+1}{2}\rceil $ operations "qsvd" on quaternion matrices  $\{\breve{\mathbf{G}}_t\}$. We remark that the computational cost of Algorithm \ref{algo:qtsvd}  mainly depends on that of "qsvd". To speed up Algorithm \ref{algo:qtsvd}, one can consider fast methods to compute quaternion SVD (see \cite{liu2022randomized} for example), or design parallel quaternion svd ("qsvd" for $\{\breve{\mathbf{G}}_t\}$ can be compute in parallel).
\end{remark}

\begin{remark}
To get the best tubal rank-$r$ approximation of quaternion tensor $\breve{\mathcal{T}}$, one only blue needs to change QSVD in step 4 in Algorithm \ref{algo:qtsvd} to truncated QSVD.
\end{remark}


\section{Numerical Examples}

In this section, we conduct experiments on  computing quaternion circulant matrix inverse, solving quaternion Toeplitz systems to test the performance of our block diagonalization results. Numerical example on color video is presented to verify the effectiveness of quaternion tensor SVD. All the experiments were run on Intel(R) Core(TM) i7-10700 CPU @2.90GHZ with 16GB of RAM using MATLAB, toolbox QTFM, and Tensorlab \footnote{https://www.tensorlab.net/}. The code is available from \url{https://github.com/Panjun009/BlkDiagCir_quaternion.git}.

\subsection{Application in Computing the Inverse of a Quaternion Circulant Matrix}
\label{Application1}
In the first application, we compute the inverse of a quaternion circulant matrix by applying our block diagonalization results of a quaternion circulant
matrix. For comparison, we use the quaternion matrix inverse function "inv'' in MATLAB toolbox QTFM. Without knowing the block-structure derived in the paper, one needs to form the quaternion circulant matrix first and then compute its inverse. While based on Theorem \ref{Thm:Sec2-1}, we are able to design a fast method to compute the inverse of a quaternion circulant matrix. 

In axis-system $(\breve{\mu},\breve{\alpha},\breve{\beta})$, given a quaternion circulant matrix $\breve{\mathbf{S}}= \text{circ}(\breve{\mathbf{s}})$, here $\breve{\mathbf{s}}$ is the first column of $\breve{\mathbf{S}}$. We denote its inverse $\breve{\mathbf{S}}^{-1}$ as $\breve{\mathbf{Z}}$. Since the inverse of a quaternion circulant matrix is also quaternion circulant matrix, we only need to compute its first column. Let $\breve{\mathbf{z}}$ be the first column of $\breve{\mathbf{Z}}$, and $\bm{\breve{\sigma}}$ be the vector by applying quaternion Fourier transform on $\breve{\mathbf{z}}$, i.e., $
\breve{\bm{\sigma}} =\text{qfft}(\breve{\mathbf{z}},\breve{\mu}, \text{'L'}).
$ The inverse $\breve{\mathbf{Z}}$ can be given by  $\breve{\mathbf{Z}}= \text{circ}(\breve{\mathbf{z}})$.  From Theorem \ref{Thm:Sec2-1}, $\breve{\mathbf{F}}_{\breve{\mu}}\breve{\mathbf{S}}\breve{\mathbf{F}}^*_{\breve{\mu}}=\breve{\mathbf{\Lambda}}$, where $\breve{\mathbf{\Lambda}}$ has the block structure \eqref{1}. We can get that 
$\breve{\mathbf{F}}_{\breve{\mu}}\breve{\mathbf{Z}}\breve{\mathbf{F}}^*_{\breve{\mu}}=\breve{\mathbf{\Lambda}}^{-1}$, and $\breve{\mathbf{\Lambda}}^{-1}$ has the same block structure as $\breve{\mathbf{\Lambda}}$. Below we present a fast method to obtain $\breve{\mathbf{z}}$.  
\begin{enumerate}

\item We apply quaternion Fourier transform on the first column $\breve{\mathbf{s}}$ of $\breve{\mathbf{S}}$, i.e., 
$
\breve{\bm{\lambda}} =\text{qfft}(\breve{\mathbf{s}},\breve{\mu}, \text{'L'}).
$

\item  According to the block structure \eqref{1}, we construct  $1\times 1$ blocks and  $2\times 2$ blocks 
  from the main diagonal and the anti lower sub-diagonal of  $\breve{\mathbf{\Lambda}}$ whose entries are given by $ \breve{\bm{\lambda}}$. In other words, the blocks are formed by the entries in $ \breve{\bm{\lambda}}$.
  
\item We compute the inverse matrices of  $1\times 1$ blocks and  $2\times 2$ blocks by their closed-form. 

\item   Inversely to step 2, we construct $\breve{\bm{\sigma}}$ by the inverse matrices obtained at Step 3 based on the block structure of $\breve{\mathbf{Z}}$ and the results in  Theorem  \ref{Thm:Sec2-1}.  Then $\breve{\mathbf{z}}$ is given by applying quaternion inverse Fourier transform on $\breve{\bm{\sigma}}$, i.e., $\breve{\mathbf{z}}=\text{iqfft}(\breve{\bm{\sigma}},\breve{\mu}, \text{'L'})$. 
 
\end{enumerate}  

To verify the effectiveness of the fast method, we generate a quaternion circulant matrix and compute its inverse. The entries of the first column of a quaternion circulant matrix are generated uniformly at random by the "randq'' function in MATLAB toolbox QTFM.  
We generate 25 quaternion circulant matrices of each dimension, and report the average time in Fig. \ref{Figs.cirInv}. The distance between an identity matrix and the product of the quaternion circulant matrix and its inverse is calculated as follows:
$$
\text{distance}=\max\{\mathbf{I}-\breve{\mathbf{S}}\breve{\mathbf{Z}},\mathbf{I}-\breve{\mathbf{Z}}\breve{\mathbf{S}}\}.
$$ 
Here $ \breve{\mathbf{Z}}$ is the inverse computed by the 
computational methods. For simplicity, we refer our proposed 
fast method as fast-circulant-inverse, and the method using inverse function as "inv" in Fig. \ref{Figs.cirInv}.

\begin{figure} [H]
\center
\includegraphics[width=0.5\textwidth]{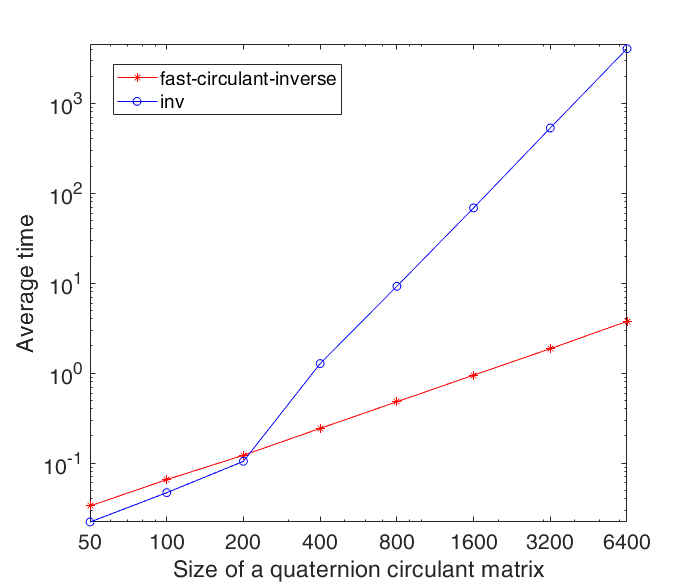}\includegraphics[width=0.5\textwidth]{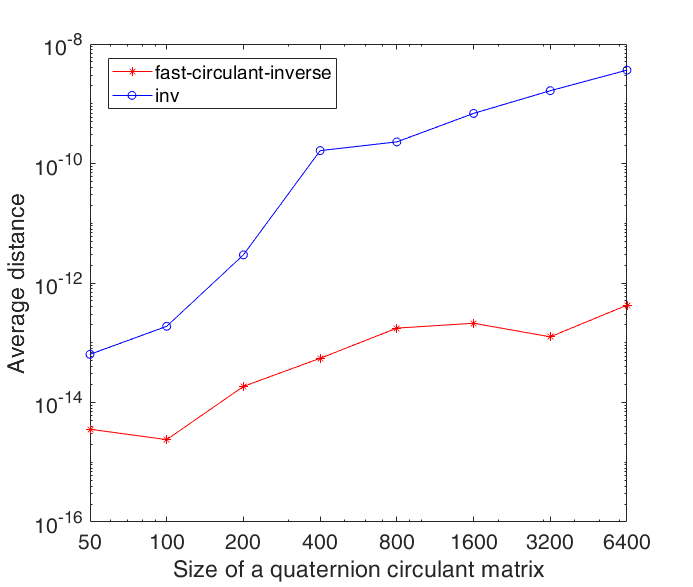}
\caption{The results of the inverse of quaternion circulant matrix.}\label{Figs.cirInv}
\end{figure}

According to Fig. \ref{Figs.cirInv}, we observe that the inverse of
quaternion circulant matrix is computed efficiently by our proposed fast method in $O(n \log n)$ operations, which is much faster than the method directly using MATLAB inverse function "inv''. From the distance metric, the inverse by the fast method is more accurate than the one obtained by the "inv'' function.  
Moreover, we remark that the proposed fast method is implemented on quaternion vectors. So we do not need to store the entire circulant matrix, but only the corresponding vectors, which can save lots of storage space. 

\subsection{Application in Solving Quaternion Toeplitz Systems}

In this section, we will study the quaternion Toeplitz matrix system by applying the preconditioned conjugate gradient method with quaternion circulant preconditioner. 
In quaternion signal processing \cite{navarro2013prediction,took2010quaternion}, we often need to estimate the transmitted quaternion signal from a sequence of received quaternion signal samples or to model an unknown system by using a linear system model. Let $\breve{x}_t$ be a discrete-time wide-sense stationary zero-mean quaternion-valued process. A linear predictor of order $n$ is given by the form
$$
\breve{y}_t = \sum_{s=1}^{n} \breve{x}_{t-s} \breve{a}_s, 
$$
where $\breve{y}_t$ is the predicted value 
based on the quaternion data $\{ \breve{x}_s \}_{s=t}^{t-n}$, and $\{ \breve{a}_s \}_{s=1}^n$ are the quaternion predictor coefficients.  The prediction error of order $n$ is defined as the difference between the actual value $\breve{x}_t$ and the
predicted value $\breve{y}_t$. Hence the 
predictor coefficients $\{ \breve{a}_s \}_{s=1}^n$
should be chosen to make the prediction error as small as possible. Similar to the linear system of equations in the complex number field  \cite{Giordano1985},  by minimizing the prediction error in the least squares sense, the optimal least squares predictor coefficients are given by the solution 
of the linear system of equations:
\begin{equation} \label{Rr}
\breve{\bf R} \breve{\bf a} = \breve{\bf r},
\end{equation}
where 
$$
\breve{\mathbf{R}} = 
\left(
           \begin{array}{ccccc}
            \breve{r}_{0} &
            \breve{r}_{1} &\cdots & \breve{r}_{n-2} &
            \breve{r}_{n-1} \\
            \breve{r}_{1}^* &
            \breve{r}_{0} & \cdots 
            &  &
            \breve{r}_{n-2} \\
            \vdots& \ddots &\ddots& \ddots & \vdots\\
						\breve{r}_{n-2}^* & 
            & \cdots & 
            \breve{r}_{0} & \breve{r}_{1} \\
            \breve{r}_{n-1}^* &
            \breve{r}_{n-2}^* &\cdots& \breve{r}_{1}^* & 
            \breve{r}_{0} 
           \end{array}
         \right), \quad
 \breve{\mathbf{r}} = 
\left(
           \begin{array}{c}
            \breve{r}_{1}^* \\
            \breve{r}_{2}^* \\
             \ddots  \\
            \breve{r}_{n-1}^* \\
            \breve{r}_{n}^* \\
            \end{array}
            \right ),
$$
and $\breve{r}_{s-\ell} = {\cal E}[ \breve{x}_{t-\ell}^* \breve{x}_{t-s} ]$, here ${\cal E}(\cdot)$ is the expectation operator. It is easy to deduce that $\breve{r}_0$ is a real number, and $\breve{r}_{s-\ell}^* = \breve{r}_{\ell-s}$. We notice that $\breve{\mathbf{R}}$ is a Hermitian quaternion Toeplitz matrix. Similar to solving complex Toeplitz matrix system, we can adopt preconditioned conjugate gradient method (PCG) to solve the equation \eqref{Rr} efficiently. More precisely, we use quaternion circulant matrix $\breve{\mathbf{S}}$ to precondition quaternion Toeplitz system \eqref{Rr} by solving the following preconditioned system instead,
\begin{equation} \label{preRr}
\breve{\mathbf{S}}^{-1}\breve{\bf R} \breve{\bf a} = \breve{\mathbf{S}}^{-1}\breve{\bf r}.
\end{equation}
For complex Toeplitz matrix system, there are many different choices of circulant preconditioners. In this paper, we consider T. Chan's circulant preconditioner \cite{chan1988optimal} in quaternion field, i.e., the $t$-th entry of $\breve{\mathbf{s}}$ that generates circulant matrix $\breve{\mathbf{S}}=\text{circ}(\breve{\mathbf{s}})$ is given by
\begin{equation}\label{TC-pre}
\breve{\mathbf{s}}_t=\left\{
\begin{array}{cc}
\frac{(n-t)\breve{\mathbf{r}}^*+t\breve{\mathbf{r}}_{n-t}}{n} &0\leq t< n,\\
\breve{\mathbf{s}}_{n+t}& 0<-t< n.
\end{array}
\right.
\end{equation}
However, in general no priori knowledge about auto-covariance of the process is provided in practice. In other words, $\breve{\bf R}$ is unknown.  While if we take $M$-data samples $\{\breve{x}_k\}^{k=M}_{k=1}$, we can still estimate the auto-covariance matrix $\breve{\mathbf{R}}$ from the data samples $\{\breve{x}_k\}^{k=M}_{k=1}$ to formulate a least squares prediction problem.  There are various types of windowing methods to estimate the auto-covariance matrix $\breve{\mathbf{R}}$, for instance, the correlation, covariance, pre-windowed and post-windowed methods, see \cite{ng1994fastiterative,ng2004iterative}. 

Let $\{\breve{x}_1,\cdots,\breve{x}_M\}$ be the set of data samples. For simplicity, we form the data matrix $\breve{\mathbf{T}}$ from the data samples with correlation windowing method by assuming that the data prior to $k=0$ and after $k=M$ are zero. Now the least squares estimation $\breve{\mathbf{a}}$ can be obtained by solving 
\begin{equation}\label{lsq}
\min\|\breve{\mathbf{T}}_w\breve{\mathbf{a}}-\breve{\mathbf{y}}\|_2,
\end{equation}
where $\breve{\mathbf{T}}_w\in \mathbb{H}^{(M+n-1)\times n}$ is a rectangular Toeplitz matrix given by 
$$
\breve{\mathbf{T}}_w=\left(
           \begin{array}{ccc}
            \breve{x}_{1} &
            ~ &~ \\
            \vdots&\ddots
              & ~\\ 
			 \breve{x}_{n} & \cdots & \breve{x}_{1} \\
            \vdots &\ddots &\vdots\\
               \vdots &\ddots &\vdots\\
           \breve{x}_{M} & \cdots   & \breve{x}_{M-n+1} \\      
           ~ &\ddots&\vdots\\
           ~ &~&\breve{x}_{M}\\
           \end{array}
         \right). 
$$ 
Therefore, the least squares solutions to \eqref{lsq} can be obtained by solving the following equation,
\begin{equation}\label{Topsyt}
\frac{1}{M}(\breve{\mathbf{T}}^*_w\breve{\mathbf{T}}_w)\breve{\mathbf{a}}=\frac{1}{M}\breve{\mathbf{T}}^*_w\breve{\mathbf{y}}.
\end{equation}
We remark that $\frac{1}{M}(\breve{\mathbf{T}}^*_w\breve{\mathbf{T}}_w)$ is a Hermitian quaternion Toeplitz matrix which can be regarded as an approximation of $\breve{\mathbf{R}}$, more precisely, $\breve{r}_t=\frac{1}{M}\sum\limits^{M-|t|}_{l=1}\breve{x}^*_l\breve{x}_{l+|t|}$. Now the solution $\breve{\mathbf{a}}$ can be solved by preconditioned conjugate gradient method with quaternion circulant preconditioner given in equation  \eqref{TC-pre}.  It is known that the PCG requires calculating the product of the quaternion circulant preconditioner's inverse and a quaternion vector. Since the inverse of the quaternion circulant matrix can be calculated by the fast method introduced in Section \ref{Application1}, the product of the inverse of the quaternion circulant preconditioner and a quaternion vector can be solved efficiently, without multiplying the entire inverse matrix by the quaternion vector.

Next, to test the effectiveness of PCG with quaternion circulant preconditioner in solving the Toeplitz system \eqref{Topsyt}, we 
consider the first order and second order autoregressive processes, i.e.,
\begin{eqnarray*}
\text{AR}(1)&:& \quad \breve{x}_t= \rho \breve{x}_{t-1}+\breve{v}_t,\\
\text{AR}(2)&:& \quad \breve{x}_t+\tau_1 \breve{x}_{t-1}+\tau_2 \breve{x}_{t-2}=\breve{v}_t,
\end{eqnarray*}
where $\{\breve{v}_t\}$ is a white noise process with variance $\eta^2$, number $\rho$ and $(\tau_1,\tau_2)$ are parameters of AR(1) and AR(2) respectively. 

In this numerical example, we generate $M=m*n$ samples $\{\breve{x}_t\}^M_{t=1}$ from the AR(1) with $\rho=0.3,0.9,0.99$, and from the AR(2) with $(\tau_1,\tau_2)=(0.1,0.5), (0.9,0.5),(0.99,0.99)$ respectively. The input of AR(1) and AR(2) are generated uniformly at random by the "randq'' function in MATLAB toolbox QTFM. The white noise process $\{\breve{v}_t\}$ is generated with variance $\eta^2$ equals $1$. We formulate the least squares prediction system \eqref{Topsyt} by the correction windowing method. For each set of parameters, we generate 25 such systems. The stopping criterion for preconditioned conjugate gradient method is set to be $\|\breve{\mathbf{e}}_k\|_2/\|\breve{\mathbf{e}}_0\|_2< 10^{-7}$, here $\breve{\mathbf{e}}_k$ is the residual vector after $k$ iterations. We employ the circulant preconditioner \eqref{TC-pre} for the preconditioned system and  report  the average iterations and computational time in Table \ref{tab:AR1} and Table \ref{tab:AR2}. For comparison, we solve the original system by the conjugate gradient method (CG) which can be seen as PCG with identity matrix $\mathbf{I}_n$ as its preconditioner. From Table \ref{tab:AR1} and Table \ref{tab:AR2}, we observe that:
\begin{itemize}
\item In terms of iterations, solving the preconditioned system requires much fewer iterations than solving the original system. As $n$ increases, the number of iterations to solve the original system increases much faster than the number of iterations to solve the preconditioned system. This phenomenon is significant when the parameter set $\rho$ in AR(1) or $(\tau_1,\tau_2)$ in AR(2)  is closer to 1.
\item In terms of computational time, solving the preconditioned system is faster than solving the original system. The advantage of preconditioned system is obvious especially when the parameter set $\rho$ in AR(1) or $(\tau_1,\tau_2)$ in AR(2) is getting closer to 1.
\end{itemize}

In summary, quaternion 
circulant preconditioner is effective in solving quaternion Toeplitz 
systems. Our theoretical results show that the inverse of an 
$n$-by-$n$ quaternion
matrix can be computed in $O(n \log n)$ operations, and therefore
the preconditioned conjugate gradient method is quite efficient for solving 
quaternion Toeplitz systems arising from the prediction of quaternion
signals.

\begin{table}[h]
  \centering 
  \begin{tabular}{||c|c||c|c||c|c||c|c||c|c||}
 
  \hline
  \multicolumn{2}{||c||}{~} &\multicolumn{8}{|c||}{$\rho=0.3$} \\
\hline
 \multicolumn{2}{||c||}{n}&\multicolumn{2}{|c||}{100}&\multicolumn{2}{|c||}{200}&\multicolumn{2}{|c||}{400}&\multicolumn{2}{|c||}{800}\\
\hline
m& &$\mathbf{I}_n$&$\breve{\mathbf{S}}_n$ &$\mathbf{I}_n$&$\breve{\mathbf{S}}_n$ &$\mathbf{I}_n$&$\breve{\mathbf{S}}_n$&$\mathbf{I}_n$&$\breve{\mathbf{S}}_n$\\  
 \hline
  \multirow{2}*{2} &iteration&50 &25 &69 &27 &74 & 30&86 &33  \\
  \cline{2-10}&computational time&0.28 &0.17 & 0.42&0.23 &0.56 & 0.36&0.85 &0.60  \\ 
\hline
 \multirow{2}*{4} &iteration&39  &20  &46 &22  &52  &23  & 58 &24 \\
  \cline{2-10}&computational time&0.24  &0.15  &0.33 &0.22  &0.46  &0.33  & 0.74 &0.56\\
 \hline
   \multirow{2}*{8} &iteration&30 &16 &36 &17 &40&18 &44 &19 \\
  \cline{2-10}&computational time&0.20 & 0.13& 0.28&0.19 &0.41 &0.31 &0.66 &0.54 \\ 
\hline
 m&    &\multicolumn{8}{|c||}{$\rho=0.9$} \\
 \hline
  \multirow{2}*{2} &iteration&134 &29 &222 & 31&356 &36&503 &39 \\
  \cline{2-10}&computational time&0.64&0.18 &1.05 &0.25 &1.77 & 0.38& 2.92&0.65 \\ 
\hline
  \multirow{2}*{4} &iteration&115 & 23 & 188&26 &278  & 27 &377  &29\\
  \cline{2-10}&computational time& 0.55 & 0.16 &0.91 &0.22  & 1.43 & 0.35 & 2.20 &0.58\\
 \hline
   \multirow{2}*{8} &iter&102 & 19&158 &20 & 233&22 &313 &23 \\
  \cline{2-10}&computational time&0.51 &0.14 &0.79 &0.20 &1.25 &0.33 &1.88 &0.56 \\ 
\hline
 m&    &\multicolumn{8}{|c||}{$\rho=0.99$} \\
 \hline
  \multirow{2}*{2} &iteration&165 &31  & 344 & 37 &657 &40 & 1282&44  \\
 \cline{2-10}&computational time&  0.75&0.19 & 1.56 &0.27  & 3.03 &0.40 &6.26 & 0.64 \\
\hline
  \multirow{2}*{4} &iteration&149 &26  & 286 &28  & 547 &31 &1022  & 33 \\
  \cline{2-10}&computational time& 0.69 &0.17  &1.31 & 0.23 &2.56  &0.36 & 5.05&0.59  \\
 \hline
   \multirow{2}*{8} &iteration&137 &22  &257  &23  &474  &25 &861 &27  \\
  \cline{2-10}&computational time&0.64 &  0.15& 1.19 &0.21  & 2.28 &0.33 & 4.33&0.57  \\
\hline   
 \end{tabular}
 \caption{Average number of iterations and computational time (in seconds) for AR(1) process }\label{tab:AR1}
\end{table}

\begin{table}[h]
  \centering 
  \begin{tabular}{||c|c||c|c||c|c||c|c||c|c||}
  \hline
  \multicolumn{2}{||c||}{~} &\multicolumn{8}{|c||}{$\tau_1=0.1$,\quad $\tau_2=0.5$} \\
\hline
 \multicolumn{2}{||c||}{n}&\multicolumn{2}{|c||}{100}&\multicolumn{2}{|c||}{200}&\multicolumn{2}{|c||}{400}&\multicolumn{2}{|c||}{800}\\
\hline
m& &$\mathbf{I}_n$&$\breve{\mathbf{S}}_n$ &$\mathbf{I}_n$&$\breve{\mathbf{S}}_n$ &$\mathbf{I}_n$&$\breve{\mathbf{S}}_n$&$\mathbf{I}_n$&$\breve{\mathbf{S}}_n$\\  
 \hline
  \multirow{2}*{2} &iteration& 66&26 &86 &29 &107 & 32&129 &34\\
  \cline{2-10}&computational time&0.35 &0.17 &0.50 &0.24 & 0.69 &0.36 &1.04&0.60  \\ 
\hline
  \multirow{2}*{4} &iteration&52 &21 &67 &22 &80 &24 &92&25 \\
  \cline{2-10}&computational time&0.29 &0.15 &0.43 &0.22 & 0.58&0.33 &0.87&0.55 \\
 \hline
   \multirow{2}*{8} &iteration&46 &17 &56 & 18&65 & 19&71 &20 \\
  \cline{2-10}&computational time&0.27 &0.14 & 0.37&0.20 &0.52 & 0.31&0.78 &0.54 \\ 
\hline
 m&    &\multicolumn{8}{|c||}{$\tau_1=0.9$,\quad $\tau_2=0.5$} \\
 \hline
  \multirow{2}*{2} &iteration&112 &27 &163 &30 &203 & 33&248 &36 \\
  \cline{2-10}&computational time& 0.55&0.18 &0.80 &0.24 &1.10 &0.36 & 1.56&0.60 \\ 
\hline
  \multirow{2}*{4} &iteration&95 &22 &124 &24 &152 &26 &174&27 \\
  \cline{2-10}&time&0.47 &0.15 &0.64 &0.22 &0.88 &0.33 &1.24&0.57 \\
 \hline
   \multirow{2}*{8} &iteration&84 &19 &106 & 20&121 &21 & 137& 22\\
  \cline{2-10}&computational time&0.42& 0.14&0.56 & 0.20&0.75 &0.32 &1.08 &0.55 \\ 
\hline   
 m&    &\multicolumn{8}{|c||}{$\tau_1=0.99$,\quad $\tau_2=0.99$} \\
 \hline
  \multirow{2}*{2} &iteration&221&40 &458 &47 &1002 &52&1924 &55 \\
  \cline{2-10}&computational time&0.99& 0.22&2.05 &0.31 &4.46 &0.44 &8.99 &0.69 \\ 
\hline
  \multirow{2}*{4} &iteration&218 &38 &436 & 41&864 &42 & 1602&43 \\
  \cline{2-10}&computational time&0.98 &0.22 &1.94 & 0.29& 3.89&0.40 & 7.57& 0.63\\
 \hline
   \multirow{2}*{8} &iteration&212 & 34&412 & 37&763 &34 &1410 &35 \\
  \cline{2-10}&computational time&0.94 & 0.20&1.91 &0.28 & 3.45&0.37 &6.77 & 0.67\\ 
\hline   
 \end{tabular}
 \caption{Average number of iterations and computational time (in seconds) for AR(2) process}\label{tab:AR2}
\end{table}

\subsection{Application in Color Image Reconstruction}

Quaternions have been widely used in image processing field. For instance, it can well represent color images  \cite{le2003quaternion, Pei, pei1999color,sangwine1996fourier},  spectro-polarimetric images and polarized images \cite{flamant2020quaternion,gil2022polarized,pan2022separable}.
In this section, we conduct experiments on color video to test the performance of the proposed quaternion T-SVD. 

Given a quaternion tensor $\breve{\mathcal{T}}\in \mathbb{Q}^{n_1\times n_2\times m}$, 
$$
\breve{\mathcal{T}}= \mathcal{T}_0 +\mathcal{T}_1 \mathtt{i}+\mathcal{T}_2 \mathtt{j}+\mathcal{T}_3 \mathtt{k},
$$
where $\mathcal{T}_l\in \mathbb{R}^{n_1\times n_2\times m}$, $l=0,1,2,3$.
We compare quaternion T-SVD (QT-SVD) with the following three methods.

\begin{itemize}

\item We apply standard Fourier transform along each tube of $\breve{\mathcal{T}}$ to obtain a tensor $\breve{\mathcal{D}}$. Then quaternion SVD is used on each frontal slice of $\breve{\mathcal{D}}$. That is,
$t=1,2,\cdots,m$,
$$
\breve{\mathcal{D}}(:,:,t)=\breve{\mathbf{U}}_t\mathbf{\Sigma}_t\breve{\mathbf{V}}^*_t.
$$
Each front slice of the rank-$r$ approximation $\breve{\mathcal{D}}^{[r]}$ is given by
$$\breve{\mathcal{D}}^{[r]}(:,:,t)\doteq\sum\limits^r_{p=1}\breve{\mathbf{U}} _t(:,p)\mathbf{\Sigma}_t(p,p)\big(\breve{\mathbf{V}}_t(:,p)\big)^*.$$
The rank-$r$ approximation $\breve{\mathcal{T}}^{[r]}$ of tensor $\breve{\mathcal{T}}$ is obtained by applying inverse Fourier transform along each tube of  $\breve{\mathcal{D}}^{[r]}$. We refer this method to as the Fourier transform based quaternion tensor factorization (FT-QTF). 

\item t-QSVD\cite{zheng2023approximation}. The t-QSVD method aims to find approximation for a third order quaternion tensor based on T-product of third order quaternion tensors given in \cite{zheng2022block}. Similar to the FT-QTF, the t-QSVD method only involves Fourier transform and quaternion SVD.

\item Standard T-SVD\cite{kilmer2011factorization,kilmer2013third}. We apply standard T-SVD on each component of the quaternion tensor, that is $$\mathcal{T}_l=\mathcal{U}_l \star \mathcal{S}_l\star \mathcal{V}_l, \quad l=0,1,2,3.$$  
Its rank-$r$ approximation $\mathcal{T}^{[r]}_l=\sum\limits^r_{p=1}\mathcal{U}_l(:,p,:)\star \mathcal{S}_l(p,p,:)\star \mathcal{V}_l(:,p,:)^T$. 
The rank-$r$ approximation of 
$\breve{\mathcal{T}}$ denoted as $\breve{\mathcal{T}}^{[r]}$ is given by  
$\breve{\mathcal{T}}^{[r]}=\mathcal{T}^{[r]}_0 +\mathcal{T}^{[r]}_1 \mathtt{i}+\mathcal{T}^{[r]}_2 \mathtt{j}+\mathcal{T}^{[r]}_3 \mathtt{k}.$
\end{itemize}


In the following, we test all the methods on "Mobile'' YUV sequences video  $\footnote{The YUV sequences video data set is downloaded from 
http://trace.eas.asu.edu/yuv/index.html}$ which contains 300 frames. Each frame of the video is a $144\times 176$ RGB image. The color value of each pixel is encoded in a pure quaternion, that is, a pixel value at location $(s,p)$ is given by 
$\breve{m}_{sp}=R_{sp}\tt{i}+G_{sp}\tt{j}+B_{sp}\tt{k}$, where $R,G,B$ denote the red, green and blue components of each pixel respectively. The quaternion representation of RGB image were proposed by Pei \cite{Pei} and Sangwine \cite{sangwine1996fourier}. Hence the video data can be represented by a pure quaternion tensor $\breve{\mathcal{T}}\in \mathbb{Q}^{144\times 176 \times 300}$.  

To show the representation ability of the methods, we  report the  average peak signal-to-noise ratio (psnr) for
the frames from the reconstructed video with the frames from the original video, denoted as
\begin{equation}\label{error}
\text{PSNR}= \frac{\sum\limits^m_{t=1}\text{psnr}\big(\breve{\mathbf{T}}^{[r]}_t,\breve{\mathbf{T}}_t\big)}{m},
\end{equation}
where $\breve{\mathbf{T}}_t= \breve{\mathcal{T}}(:,:,t)$, $\breve{\mathbf{T}}^{[r]}_t = \breve{\mathcal{T}}^{[r]}(:,:,t)$,  and "$\text{psnr}$" is the peak signal-to-noise ratio computed by using "psnr" function in MATLAB.

We also present the average structural similarity index between  the frames from the reconstructed video and the frames from the original video, that is 
\begin{equation}\label{ssim}
\text{SSIM} = \frac{\sum\limits^m_{t=1}\text{ssim}\big(\breve{\mathbf{T}}^{[r]}_t,\breve{\mathbf{T}}_t\big)}{m},
\end{equation}
here "$\text{ssim}$" is the structural similarity index value computed by using "ssim" function in MATLAB. 

Instead of applying the methods directly on $\breve{\mathcal{T}}$, we preprocess the video set by "image mean subtraction". That is, we first compute the mean of all the frames, and then subtract the mean image from all frames. All the methods are used in the processed data. We report the average PSNR and SSIM  (in percentage)  in Table \ref{tab:video}. The average computational time (in seconds) is also presented in the table. 
 
\begin{table}[h]
  \centering 
  \begin{tabular}{|c||c|c||c|c||c|c||c|c||c|}
\hline
 \multirow{2}*{} &\multicolumn{2}{|c||}{$r=10$}&\multicolumn{2}{|c||}{$r=20$}&\multicolumn{2}{|c||}{$r=40$}&\multicolumn{2}{|c||}{$r=80$} &\multirow{2}*{time} \\
\cline{2-9} &PSNR&SSIM  &PSNR&SSIM &PSNR&SSIM&PSNR&SSIM&\\  
 \hline
QT-SVD& \textbf{19.53} &\textbf{ 74.42} &\textbf{22.83} &\textbf{85.62} &\textbf{ 27.89}&\textbf{94.56}&\textbf{37.46} &\textbf{99.24}&518.03\\
\hline
FT-QTF&19.01 & 72.67 &22.36 &84.63&27.53 &94.19&37.29&99.21&305.77  \\
\hline
t-QSVD&19.01 &72.67 &22.36 &84.63 &27.53 &94.19 &37.29&99.21 &191.49\\
 \hline
 T-SVD&19.43&74.03& 22.70&85.30 &27.77 &94.41&37.25 &99.20&\textbf{39.41}\\ 
\hline
 \end{tabular}
 \caption{The reconstruction results of "Mobile" video}\label{tab:video}
\end{table} 
 
From the experimental results, we have the following observations.
\begin{itemize}

\item In terms of PSNR, we can see the value increases as the increase of the value of $r$. The values of PNSR of FT-QTF and t-QSVD are less than those from the other two methods when $r\leq 40$. It is interesting to see that the values of PSNR from all the quaternion-based methods are better than T-SVD when $r=80$. Our QT-SVD has the highest values among all the $r$. It  implies that our method performs the best. 

\item In terms of SSIM, we have similar observation to that of PSNR. The value of SSIM increases  when $r$ gets larger. The values of the QT-SVD are all higher than the others methods. 

\item Regarding the average computational time,  the proposed QT-SVD requires more time than the other methods, followed by FT-QTF and t-QSVD. It is not surprising that the computational time of T-SVD is the least since the quaternion SVD used in the other three methods is much slower than the SVD.

\item We also notice that the results of t-QSVD and FT-QTF are the same. The reason is that the computational principles behind both methods are the same, i.e., applying the standard Fourier transform along the tubes of the input quaternion tensor. 
\end{itemize}

To show the effect of image reconstruction visually, we simply choose the $141$-th - $147$-th frames 
of the video and present their reconstruction results of the methods for $r=20$ in Fig. \ref{mobile.r20}. To better see the results, the "psnr" value is given below each reconstructed frame. Among all the reconstructed frames, the frames reconstructed by our QT-SVD are the best. The "psnr" values are higher than those by the other methods. We also observe that the reconstructed results from the t-QSVD and FT-QTF are the same, due to the same computational principles. We also exhibit the details of different regions of the $147$-th frame in Fig. \ref{mobiledel.r20}. Compared to the other methods,  our QT-SVD  can capture better details, and the color of reconstructed image is closer to the real image, see the sheep's faces, those orange flowers in Fig. \ref{mobiledel.r20} for instance.
\begin{figure}[h]
\includegraphics[width=1\textwidth]{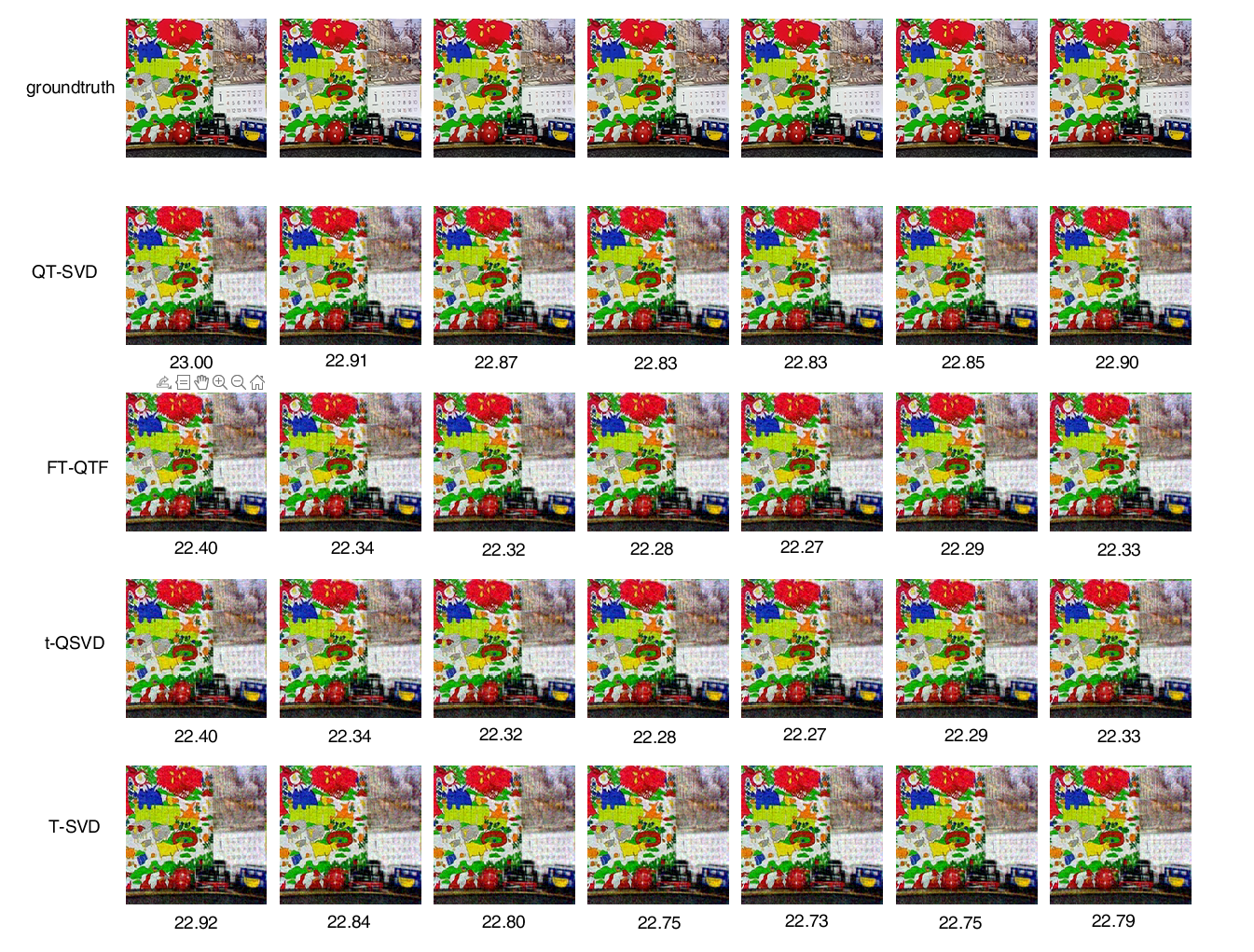}
\caption{The reconstruction results of "Mobile" at frame $141$-th - $147$-th by the methods when $r=20$. The first row represents the  frames in original video.  The "psnr" value is given below each reconstructed frame. }\label{mobile.r20}
\end{figure}

\begin{figure}[h]
\includegraphics[width= \textwidth]{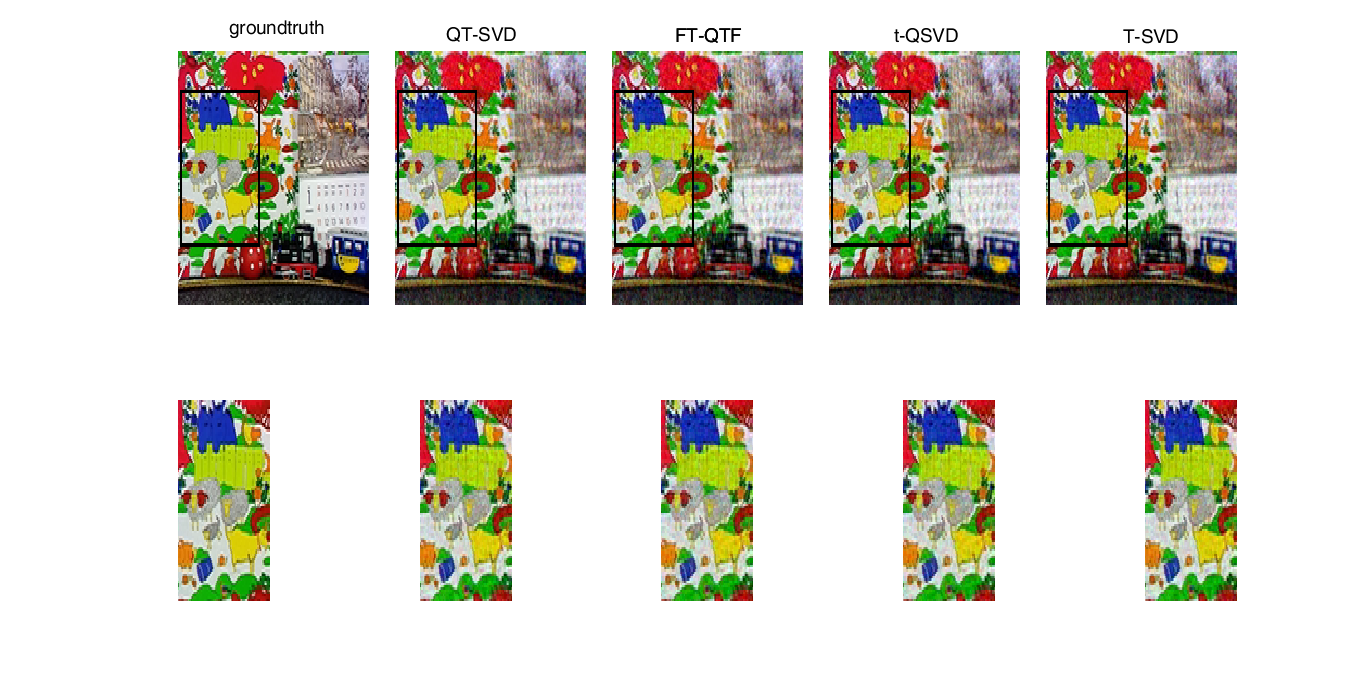}
\caption{The top row shows the reconstructed images of the $147$-th frame of "Mobile" when $r=20$. The bottom row shows the details of the specific area of images on the top row  framed in black.}\label{mobiledel.r20}
\end{figure}

\section{Conclusion}

The paper studied quaternion circulant matrix  and proved that any circulant matrix can be block-diagonalized into 1-by-1 block and 2-by-2 block by discrete quaternion Fourier transform matrix. In other words, discrete quaternion Fourier transform matrix
is not universal eigenvectors for quaternion circulant matrices. Indeed, the eigenvalues 
and their associated eigenvectors of quaternion circulant matrices can be obtained by 
the combination of discrete quaternion Fourier transform matrix and 
quaternion matrices that can diagonalize 2-by-2 block structure of the 
quaternion transformed circulant matrices.
The results are used to studied quaternion tensor singular value decomposition which is based on the well-known T-SVD form.  We tested and showed the proposed block diagonalization
results of circulant matrices for computing quaternion circulant matrix inverse, solving linear prediction of quaternion signal processing. An example of color video is used to demonstrate the effectiveness of quaternion tensor singular value decomposition.



\end{document}